\documentclass[11pt,reqno]{amsart}
\usepackage{amsmath,amsthm,amsfonts,amssymb,mathrsfs,bm,graphicx,stmaryrd}
\usepackage{pkgfile}
\usepackage{hyperref}
\usepackage[usenames]{color}
\usepackage[letterpaper,hmargin=1.0in,vmargin=1.0in]{geometry}

\parskip 	\smallskipamount
\def\N{\mathbb{N}}
\def\P{\mathbb{P}}
\def\Z{\mathbb{Z}}
\def\R{\mathbb{R}}

\def\E{\mathbb{E}}
\begin{document}
\title[$SO(N)$ Lattice Gauge Theory]{ $SO(N)$ Lattice Gauge Theory, planar and beyond}

\author{Riddhipratim Basu}
\address{Riddhipratim Basu, Department of Mathematics, Stanford University, Stanford, CA, USA}
\email{rbasu@stanford.edu}
\author{Shirshendu Ganguly}
\address{Shirshendu Ganguly, Department of Statistics, UC Berkeley, Berkeley, CA, USA}
\email{sganguly@berkeley.edu}
%\thanks{SG's research is supported by a Miller Research Fellowship at UC Berkeley.}
%\thanks{Part of this work was completed while the second author was a Miller Fellow at UC Berkeley.}
\date{}

\begin{abstract}
Lattice Gauge theories have been studied in the physics literature as discrete approximations to quantum Yang-Mills theory for a long time. Primary statistics of interest in these models are expectations of the so called ``Wilson loop variables". In this article we continue the program initiated by Chatterjee \cite{Cha15} to understand Wilson loop expectations in Lattice Gauge theories in a certain limit through gauge-string duality. The objective in this paper is to better understand the underlying  combinatorics in the strong coupling regime, by giving a more geometric picture of string trajectories involving correspondence to objects such as decorated trees and non-crossing partitions.
Using connections with Free Probability theory, we provide an elaborate description of loop expectations in the planar setting, which provides certain insights about structures of higher dimensional trajectories as well. Exploiting this, we construct an example showing that in any dimension, the Wilson loop area law lower bound does not hold in full generality.
%and thereby partially answer a question from \cite{Cha15} in the negative. 
\end{abstract}

\maketitle
\setcounter{tocdepth}{1}
%\tableofcontents

\section{Introduction}
Matrix integrals are known to provide canonical models for generating family of combinatorial objects relevant to studying physical systems \cite{thooft, zvonkin}. The connection of Gaussian integral with enumeration of maps has been classically studied \cite{BIPZ}, and it has been believed \cite{thooft} that similar topological expansions should hold for more general models invariant under unitary conjugation. Much of this theory has become mathematically well-founded due to extensive work by Guionnet and coauthors \cite{G04,GM07,CGM09} in the last decade where asymptotics of orthogonal and unitary matrix integrals have been studied in great detail. 

In physics literature one of the motivations for studying Gibbs measure on matrices has been to understand the so called ``Lattice Gauge theories". These were introduced by Wilson \cite{Wilson} as a mathematically well-defined approximation to quantum Yang-Mills theories, the basic building blocks of the Standard Model of quantum mechanics. Statistics of interest in these models are expectations (under the Gibbs measure) of certain variables called ``Wilson loop variables". Approximate computations of the the loop expectations was suggested by 't Hooft \cite{thooft}, in what came to be known as the 't Hooft limit, using connections between matrix integrals and enumeration of planar maps. As mentioned above much of this connection had now been made rigorous; however, computing formulae for loop expectations, had largely remained open until recently.

In his recent seminal work, Chatterjee \cite{Cha15} solved this problem for a Euclidean lattice Gauge theory with Gauge group $SO(N)$ in the large $N$ limit, and provided an asymptotic formula for loop expectations in terms of a ``lattice string theory", thus establishing rigorously one of the first examples of ``gauge-string duality". Chatterjee's method hinges on making rigorous, a set of equations known as ``Makeenko-Migdal equations" in physics. In a later work, Chatterjee and Jafarov \cite{CJ16}, generalized this result and proved a $\frac{1}{N}$ expansion of the loop expectations under strong coupling. We refer the interested reader to \cite{Cha15, CJ16, CGM09} for more background on this. 

\noindent \textbf{Our contributions:}   This article begins by giving a complete combinatorial description of the planar model using the above machinery. The expression Chatterjee obtains for the loop expectation in 't Hooft limit is given, under strong coupling, by a power series in the inverse coupling constant $\beta$. Our work starts with the observation that one can explicitly compute the power series in dimension two,  for a large class of loops.  

For $SO(N)$ lattice gauge theory on the plane, the structure of the Hamiltonian turns out to be invariant under conjugation by elements of $O(N)$.  Such invariance properties allow us to show asymptotic freeness of the underlying matrices and thereby use many combinatorial  tools and identities from Free Probability theory involving objects such as non-crossing partitions. These are used to  to analyze expectations of  rather complicated loop variables. 
Using the notion of free cumulants, one can show in fact that  the power series expansions (in $\beta$) of the loop expectations, contain only finitely many terms, i.e. they are polynomials.

We introduce a class of decorated trees, which are used  to give a geometric description of certain carefully rooted  version of string trajectories appearing in \cite{Cha15}.
Then exploiting Chatterjee's  gauge-string duality, we compute the loop expectation for all simple loops and show that the limiting expression is $\beta^{k}$ where $k$ is the area enclosed by the loop (see Section \ref{defresult} for formal definitions). 

Obtaining insight from the planar picture we provide a general correspondence between trajectories appearing in \cite{Cha15} and non-crossing partitions in all dimensions.
Using the above we  provide an example, showing that for certain non-simple `cancelling' loops, the loop expectations can decay faster than exponentially in the area enclosed. This provides a partial negative answer to a question of Chatterjee \cite{Cha15} in any dimension, regarding area law lower bounds for Wilson loop variables. To the best of our knowledge this is the first rigorous computational result in high dimensional lattice gauge theory. %\textcolor{red}{change wording}
\begin{remark}\label{reparam}
Soon after this work was completed, Jafarov posted \cite{JJ16} in which he proves results similar to \cite{Cha15} and \cite{CJ16} when the gauge group is $SU(N).$ In this work among other things, he establishes (see \cite[Corollary 4.4]{JJ16}) that, in the strong coupling regime, the Wilson loop expectations for inverse coupling constant $\beta$ in the $SO(N)$ theory exactly matches those in the $SU(N)$ theory for inverse coupling constant $2\beta$. Thus our main results have natural versions for the $SU(N)$ theory once the appropriate re-parametrization is done. 
\end{remark}
The exactly solvable nature of the model in two dimensions makes the model mathematically more tractable. Results analogous to some in this article, in the two dimensional $U(N)$ lattice gauge theory appear as semi-rigorous work in the physics literature (see \cite{GW80, Wadia}), where the arguments mostly rely on mean field approximations to asymptotic eigenvalue distributions.\footnote{Since the submission of this paper, it has come to our attention that the arguments in \cite{GW80, Wadia}  have been formalized in \cite{LDPunitary} by a method extremely special to the planar case and  quite different from the general string theoretic approach taken in this paper. The main steps in \cite{LDPunitary} include proving a large deviation principle for a certain class of Gibbs measures on the Unitary group and solving the associated variational problem.
}  Even though we are inspired by the approach in the above works (see `Gross-Witten trick' in the proof of Lemma \ref{varchange} in Section \ref{fa}), we emphasize that our approach in this paper is purely geometric, with the motivation to go beyond the planar setting using relations to non-crossing partitions, decorated trees etc.  We believe proper random surface analogues of these would be useful in depicting the picture in higher dimensions.

It must be  mentioned here, that another class of measures have been studied extensively with a view to build the continuum Yang-Mills theory on the plane. These are based on the heat kernel of Brownian Motion on compact Lie groups and rigorously analyzing Makeenko-Migdal equations \cite{MM} in this setting. This approach was developed simultaneously by Gross, King, and Sengupta in \cite{GKS} and  Driver in \cite{Dr}. Using the remarkable result about asymptotic free limit of the such diffusions \cite{Biane95}, Makeenko-Migdal equations for the continuum model was solved in \cite{Levy11}. Certain moment computations of a similar flavour as in this article also appear in that paper. The analysis of Makeenko-Migdal equations in this context has recently been greatly simplified in \cite{DHK16}.

\section{Definitions and Main results}\label{defresult}
We now move towards formal definitions of the model. We shall follow closely the terminology and notation introduced in \cite{Cha15}. Consider the two dimensional Euclidean lattice $\Z^2$ with nearest neighbour edges. Let $E$ denote the set of all directed edges. Consider the lexicographic ordering of vertices in $\Z^2$. Call a directed edge in $E$ positively oriented if the ending point of the edge is greater than the starting point of the edge in the lexicographic ordering. Let $E^{+}$ denote the set of all positively oriented edges. For an edge $e\in E^{+}$, we shall denote the reverse edge by $e^{-1}$. A plaquette $p=e_1e_2e_3e_4$ is a closed loop of length four containing four distinct edges. A plaquette is called positively oriented if the smallest and the second smallest vertex contained in the plaquette occur in that order. We identify plaquettes that are cyclically equivalent, i.e., $e_1e_2e_3e_4$ and $e_2e_3e_4e_1$ will be considered to be the same plaquette. 

\subsection{Gibbs Measure}
For $N\in \N$, let $SO(N)$ denote the special orthogonal group of $N\times N$ orthogonal matrices, with real entries and determinant one. Fix a finite subset $\Lambda$ of $\Z^2$. Let $E_{\Lambda}$ (resp.\ $E_{\Lambda}^{+}$) denote the set of edges in $E$ (resp.\ in $E^{+}$) with both endpoints contained in $\Lambda$. Let $\mathcal{P}_{\Lambda}$ (resp.\ $\mathcal{P}_{\Lambda}^{+}$) denote the set of all plaquettes (resp.\ positively oriented plaquettes) having all the edges in $E_{\Lambda}$. 

For $\beta\in \R$, we consider the Gibbs measure $\mu_{\Lambda, N, \beta}$ on the space of configurations $Q=(Q_{e})_{e\in E_{\Lambda}^{+}}$ of matrices in $SO(N)$ defined as follows. Let $\sigma_{N}$ denote the Haar probability measure on the group $SO(N)$. Let $\sigma_{\Lambda,N}$ denote the product Haar measure on the space of $SO(N)$ matrices indexed by edges in $E_{\Lambda}^{+}$, i.e.,
\begin{equation}\label{prodhaar}
 d \sigma_{\Lambda, N}(Q)=\prod_{e\in E_{\Lambda}^{+}} d\sigma_{N}(Q_{e}).
 \end{equation}
Define the Gibbs measure $\mu_{\Lambda, N, \beta}$ by the density, 
\begin{equation}
\label{e:gibbs}
\dfrac{d \mu_{\Lambda, N, \beta}}{d \sigma_{\Lambda, N}}(Q)=Z_{\Lambda, N, \beta}^{-1} \exp \biggl( N\beta \sum_{p\in \mathcal{P}_{\Lambda}^{+}} {\rm Tr}(Q_{p})\biggr),
\end{equation}
where $Q_p=Q_{e_1}Q_{e_2}Q_{e_3}Q_{e_4},$\footnote{As we identify both $e_1e_2e_3e_4$ and $e_2e_3e_4e_1$ as the same plaquette $p,$ the matrix $Q_p$ is not quite well defined and is defined only up to a conjugation. However \eqref{e:gibbs} only depends on $Q_p$ through its trace which is well defined. Later in this article we will be more specific about our definitions of $Q_p$ to suit our arguments. For a general loop $\ell$ we shall define the matrix $Q_{\ell}$ similarly and follow the same convention. Whenever necessary we shall specifically mention the starting point and ending point of a loop $\ell$, and the definition of $Q_{\ell}$ will be accordingly interpreted in that context.} for $p=e_1e_2e_3e_4\in \mathcal{P}_{\Lambda}^{+},$ and $Q_{e^{-1}}=Q_{e}^{-1}$ for $e\in E_{\Lambda}^{+}$, and $Z_{\Lambda, N, \beta}$ denotes the normalizing constant.
This measure describes a lattice gauge theory on $\Lambda$ for the gauge group $SO(N)$. The parameter $\beta$ is called the inverse \emph{coupling constant} of the model.

\begin{remark}Since the bulk of the paper treats the planar case we define the Gibbs measure and state most of  the results from \cite{Cha15} in two dimensions. In \cite{Cha15}, Chatterjee deals with the more general $d$ dimensional $SO(N)$ Lattice Gauge theory, where the Gibbs measure is defined exactly as in \eqref{e:gibbs} by taking $\Lambda$ to be a subset of $\Z^d.$ All results in \cite{Cha15} (natural analogues of what we quote here) are valid in all dimensions. We point out that one of the main observations in this paper holds in any dimension, (see Proposition \ref{p:lbcounterex}). 
\end{remark}

\subsection{Wilson Loops}\label{wil}
One of the primary objects of interest in lattice gauge theories are Wilson loop variables and their expected values under the Gibbs measure. 
A walk is a sequence of edges $e_1,e_2,\ldots e_n$ where the end point of $e_i,$ is the starting point of $e_{i+1}$ for $1\le i\le n-1.$
A walk is said to be closed if the end point of $e_n$ is the same as the starting point of $e_1.$ A non backtracking walk is a walk with no backtracks i.e.\ $e_i \neq e^{-1}_{i+1}$ for all $i.$

For a loop\footnote{We will also formally consider the null loop i.e., which has no edges, and use $\emptyset$ to denote it.} (non-backtracking closed walk) $\ell=e_1e_2\cdots e_n,$ the Wilson loop variable is defined by 
$$W_{\ell}={\rm Tr}(Q_{e_1}\cdots Q_{e_n}).$$ 

\begin{remark} We can obtain a non-backtracking loop starting from a closed walk by performing \emph{backtrack erasure}, i.e.,\ sequentially deleting pairs of consecutive edges that are reverses of each other. Often in this article we will call a closed walk, a loop, even though it will have backtracks.  It will be explicitly mentioned when we do so and  there will be no scope for confusion. Clearly the product of the matrices $Q_{e_i},$ along a closed walk and its backtrack erasure are the same, so considering closed walks would not affect the results. 
\end{remark}
 
\begin{defn}
Throughout this article, we will say a loop $\ell=e_1e_2\cdots e_n,$ is simple if all the endpoints of  $e_i$'s are distinct. 
\end{defn} 
By the Jordan Curve Theorem, a simple loop divides the plane into two components, one bounded and one unbounded. The bounded component is a union of unit squares which we shall often identify with the plaquettes that form their boundaries and refer to these plaquettes as \emph{plaquettes contained in the interior of a simple loop}.

For notational convenience, we also define: $Q_{\ell}:=Q_{e_1}\cdots Q_{e_n}$ for a loop $\ell$. If the edges of the loop all belong to $E_{\Lambda},$ we define $\langle W_{\ell}\rangle _{\Lambda, N, \beta}$ to be the expected value of $W_{\ell},$ under the Gibbs measure $\mu_{\Lambda, N, \beta}$. In his seminal work \cite{Cha15}, Chatterjee showed that in the strong coupling regime (i.e., when $|\beta|$ is small), the loop expectations, properly scaled, converge as $N\to \infty$. The main result of Chatterjee \cite{Cha15}, simplified to our setting is the following.

\begin{thm}[Theorem 3.1, \cite{Cha15}]
\label{t:looplimit}
Consider a sequence of subsets $\Lambda_1, \Lambda_2, \ldots$ increasing to $\Z^2$. Then there exists $\beta_0>0$ such that for $|\beta|<\beta_0$ and for all loops $\ell$, we have
$$\lim_{N\to \infty} \frac{\langle W_{\ell} \rangle _{\Lambda_{N}, N, \beta}}{N}= w(\ell, \beta),$$
exists.
\end{thm}
Chatterjee's Theorem also contains an expression of $w(\ell,\beta)$ in terms of his lattice string theory, which involves summing over weights of non-vanishing loop trajectories associated with the loop $\ell$.

One can also consider a product of Wilson loop variables for a sequence of loops, and Chatterjee proves an asymptotic factorization property for such products of loop variables.

\begin{thm}[Corollary 3.2, \cite{Cha15}] 
\label{t:factor}
In the set-up of Theorem \ref{t:looplimit}, consider a sequence of loops $(\ell_1, \ell_2,\ldots, \ell_n)$. We have
$$\lim_{N\to \infty} \frac{\langle W_{\ell_1} W_{\ell_2}\cdots W_{\ell_n} \rangle _{\Lambda_{N}, N, \beta} }{N^{n}}= \lim_{N\to \infty}\prod_{i=1}^{n} \frac{\langle  W_{\ell_i} \rangle  _{\Lambda_{N}, N, \beta}}{N}$$
for $|\beta|<\beta_0$.
\end{thm}

The expression $w(\ell, \beta)$ obtained by Chatterjee is not explicit and is expressed in terms of a lattice string theory; in particular as a sum of weights of certain strings, (see \cite{Cha15} for more details). As a consequence Chatterjee provides an alternative description of the limiting expectations of the loop variables, which will be more useful to us. In the strong coupling regime, he proves that the the limiting loop expectations are real analytic, and have an absolutely convergent power series expansion,
\begin{equation}
\label{e:series}
w(\ell,\beta)=\sum_{k=0}^{\infty} a_k(\ell)\beta^{k}.
\end{equation}
One of our main results in this article is to evaluate these power series for a sufficiently large class of loops, see Theorem \ref{t:plaquette} and Theorem \ref{t:simple} below.

\subsection{Area Law Bounds}
\label{s:area}
One question of interest in physics is to understand how the loop expectations vary with the area enclosed by a loop. To introduce the results formally we need to define the area of a loop formally through the language of 1-chain and 2-chains in cell complexes. We are again following the treatment in \cite{Cha15} and introduce the following definitions that we need. For our purposes, the 1-chains are elements of the free $\Z$-module over $E^{+}$, and the two chains are the elements of the free $\Z$-module over $\mathcal{P}^{+}$. Observe that any $p\in \mathcal{P}^{+},$ can be written uniquely as $p=e_1e_2e_3^{-1}e_4^{-1},$ where $e_1,e_2,e_3,e_4,$ are all in $E^{+}$. The standard differential map $\delta$ from the module of 2-chains to the module of 1-chains takes $p$ to $e_1+e_2-e_3-e_4$. For a loop $\ell=e_1e_2\cdots e_n,$ define $r(\ell)=\sum_{i=1}^{n} r(e_i),$ where $r(e_i)=e_i$ or $-e^{-1}_i,$ depending on whether $e_i\in E^+,$ or not. Observe that $r(\ell)$ is well defined under cyclical equivalence. Now for a 2-chain $x,$ call $\ell$ to be the boundary of $x$ if $\delta(x)=r(\ell)$.  Finally we define the area of a 2-chain $x=\sum_{p\in \mathcal{P}^{+}} \eta_{p} p$ by, 
\begin{equation}\label{area}
{\rm area}(x):=\sum_{p} |\eta_{p}|.
\end{equation}
Define the area of a loop $\ell$, denoted by ${\rm area}(\ell),$ to be the minimum of ${\rm area}(x),$ over all $x,$ such that the boundary of $x$ is $\ell$, (it follows from the standard facts about cell complexes in $\Z^2,$ that for any loop this is a well defined quantity). 

\textbf{Examples:}
\begin{enumerate}
\item For a simple loop $\ell,$ the area of $\ell,$ is simply the number of plaquettes contained in the interior of $\ell$.
\item There can be non-simple non-null loops of area $0$. For two oriented adjacent plaquettes $a,b$ ($a^{-1}$ and $b^{-1}$ denote the plaquettes in the opposite orientation). Fix a point $x$ shared by both $a$ and $b$. Now consider the loop started from $x$ denoted by  $aba^{-1}b^{-1}$ or it being repeated $k$ times. Note that after tracing out either $a,b,a^{-1},b^{-1}$ the loop is at $x.$ See Fig \ref{f:lex}.

\begin{figure}[h]
\centering
\includegraphics[width=.5\textwidth]{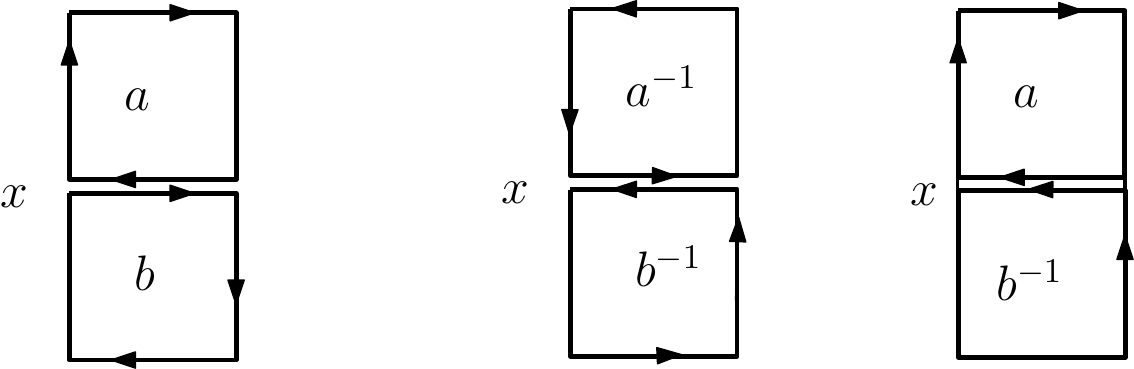}
\caption{The loop $aba^{-1}b^{-1}$ formed by tracing out adjacent plaquettes $a,b$ along different orientations yield a non-simple non-null loop of area $0.$}
\label{f:lex}
\end{figure}

It is easy to see that the boundary of $aba^{-1}b^{-1}$ is actually zero and hence the area of such a loop is zero. 
%\textcolor{red}{change wording??} 
Such examples will be analyzed later in a discussion regarding general area law lower bound for loops  (see Proposition \ref{p:lbcounterex}) according to this definition of area.
\end{enumerate}

A lattice gauge theory is said to satisfy an area law upper bound if 
$$\langle W_{\ell} \rangle \leq C_1e^{-C_2 {\rm area}(\ell)}$$
where $C_1$ and $C_2$ are constants that depend on the gauge group and the inverse coupling strength $\beta$. The theory is said to satisfy area law lower bound if the reverse inequality holds with possibly different constants, i.e.,
$$\langle W_{\ell} \rangle \geq C'_1e^{-C'_2 {\rm area}(\ell)}.$$
The area law upper bound has some connections with the theory of quirk confinement \cite{Wilson}. For the strong coupling regime in which Theorem \ref{t:looplimit} and Theorem \ref{t:factor} hold, Chatterjee also proves the area law upper bound for the $SO(N)$ lattice gauge theory in the large $N$ limit for all `non-canceling' loops\footnote{A loop is called non-canceling, if there is no edge $e$ in the loop such that $e^{
-1},$
is also in the
loop.}.

\begin{thm}[Corollary 3.3, \cite{Cha15}]
\label{t:areaub}
In the setting of Theorem \ref{t:looplimit}, one has for all non-canceling loops $\ell$,
$$\lim_{N\to \infty}\frac{|\langle W_{\ell} \rangle|}{N}\leq (C|\beta|)^{{\rm area}(\ell)},$$
for some absolute constant $C$.
\end{thm}

 Whether an area lower bound exists in any dimensions was left as an open question in \cite{Cha15}. In two dimensions, using spectral methods a general lower bound was given for rectangles for certain lattice gauge theories by Seiler \cite{Sei}. We complement this result by showing that for a large class of loops in two dimensions (including all simple loops) area law lower bound does indeed hold in the large $N$ limit, (see Corollary \ref{c:arealb}). However, by considering certain `canceling' loops, we show that the area law lower bound is not true in general in any dimension, (see Proposition \ref{p:lbcounterex}) at least with the definition of the area stated above.

\subsection{Main Results}
Our objective in this manuscript is to study the loop expectations in the large $N$ limit (known as 't Hooft limit) and explicitly evaluate the limiting loop expectations for various classes of loops in the strong coupling regime in two dimensional $SO(N)$ lattice gauge theory. {With the exception of Proposition \ref{p:lbcounterex}, all of the following results are for the case $d=2$.} Our first result computes the limiting loop expectation $w(\ell,\beta)$ in the case where $\ell$ is a plaquette.

\begin{thm}
\label{t:plaquette}
Assume the set-up of Theorem \ref{t:looplimit}, and assume $\beta$ is sufficiently small such that the conclusion in that theorem holds. Then we have for a plaquette $p,$
$$w(p,\beta)=\beta.$$
\end{thm}  

We can also obtain an explicit expression for the limit for all simple loops. 

\begin{thm}
\label{t:simple}
In the setting of Theorem \ref{t:plaquette}, for any simple loop $\ell$ of area $k$ we have,
$$w(\ell, \beta)=\beta^{k}.$$
\end{thm}

Although Theorem \ref{t:plaquette} is a special case of Theorem \ref{t:simple}, we have chosen to state it separately, as we prove Theorem \ref{t:plaquette} first using combinatorial arguments, and then prove Theorem \ref{t:simple} using a factorization that comes from fixing a suitable gauge and asymptotic freeness of Haar distributed matrices in $SO(N)$.

An immediate corollary of Theorem \ref{t:simple} is an area law lower bound for simple loops in 't Hooft limit for the strong coupling regime.

\begin{cor}
\label{c:arealb}
In the setting of Theorem \ref{t:looplimit}, for $\beta$ sufficiently small there exists $C_1, C_2$ depending on $\beta$ such that for all simple loops $\ell$ we have,
$$\lim_{N\to \infty}\frac{|\langle W_{\ell} \rangle|}{N}\geq C_1e^{-C_2 {\rm area}(\ell)}.$$
\end{cor}
Notice that this corollary is proved immediately by using Theorem \ref{t:simple} by taking $C_1=1$ and $C_2=\log |\beta|^{-1}$. As mentioned before the result generalizes the result in \cite{Sei} which proved an area law lower bound for rectangular loops (although in a slightly different setting). The next result shows however,  that the area law lower bound does not hold for all loops in dimension 2.  The proof technique gives us a way to geometrically understand string trajectories in any dimension via the theory of non-crossing partitions, even though the connection to free probability is lost.

%\red{partially answer an open question in \cite{Cha15}.}

\begin{ppn}
\label{p:lbcounterex}In the strong coupling regime (with $\beta>0$) in any dimension, there is an absolute constant $C$  depending only on the dimension such that for every $K,k>0$ there exist loops $\ell$ with area at most $k$, such that,
$$\lim_{N\to \infty}\frac{|\langle W_{\ell} \rangle|}{N}\leq (C\beta )^{K}.$$ 
\end{ppn}  
It is clear by taking a sequence of loops as given by Proposition \ref{p:lbcounterex} with $k$ fixed and $K$ increasing to infinity, that area law lower bound must fail for a sequence of such loops.  We point out that this does not rule out an area law lower bound holding according to a different definition of area than what is being used in this paper. We elaborate on this more in Section \ref{s:oq}.

Finally, our last main result shows that the power series expression for $w(\ell,\beta)$ in \eqref{e:series} must terminate, that is, $w(\ell,\beta)$ must be a polynomial. 

\begin{thm}
\label{t:poly}
In the strong coupling regime of Theorem \ref{t:looplimit}, for any loop $\ell$ there exists some $k_0=k_0(\ell)$ such that $a_{k}(\ell)=0$ for all $k>k_0$ where $a_k(\ell)$ is as in \eqref{e:series}. That is, we have, 
$$w(\ell,\beta)=\sum_{k=0}^{k_0} a_k(\ell)\beta^{k}.$$
\end{thm}

\section{Overview}
In this section, we give a broad overview of our techniques which combine a variety of combinatorial and analytical tools. Our starting point is a fundamental recursion of Chatterjee for the coefficients $a_{k}$ in \eqref{e:series}; see Section \ref{rec1}. The recursion, coming from a lattice string theory developed by Chatterjee, expresses the $k$-th coefficient $a_{k}(\ell)$ of the power series of a loop $\ell$, in terms of the $k$-th and smaller coefficients of certain functions of the loop $\ell$, called ``splitting" and ``deformation"; see the next section for formal definitions of these. Using this recursion and some combinatorial analysis, we are able to establish Theorem \ref{t:plaquette}; that is to show that the power series of a single plaquette is $\beta$. It should be possible to take this analysis further and prove Theorem \ref{t:simple} by this argument, but an observation already present in the physics literature \cite{Wadia, GW80}, simplifies that task. We describe the notion of axial gauge fixing formally later in the article (see Section \ref{gf}), but informally this refers to fixing the values of matrices corresponding to certain edges, without changing the law of the statistics of interest. By gauge fixing, one can associate independent orthogonal matrices to each of the plaquettes with a certain distribution, and one can then exploit the asymptotic freeness (see below for the relevant definitions) of those matrices to show that the loop expectations factorize in a certain sense, in the 't Hooft limit. Using the standard moment cumulant formulae from free probability theory, one can then obtain an expression of the limiting loop expectation $w(\ell, \beta)$ in terms of some standard combinatorial objects.

%\textcolor{red}{this overview here needs to change}.
 
\subsection{Elements of Chatterjee's Lattice String Theory}\label{cst}
In \cite{Cha15}, Chatterjee developed a lattice string theory by defining certain operations on loops: ``splitting", ``merger", ``deformation" and ``twisting", which are analogues of standard operations of string theory in the continuum setting. These are operations on loop(s), which produce one or more different loops. We shall not recall all the details of the formal definitions that go into this construction, and only recall the bare essentials needed for our purpose (the interested reader is referred to Section 2.2 of \cite{Cha15} for more details). We start with the definitions of some operations on loops. In what follows, $[\ell]$ will define the \emph{backtrack erasure} of the closed walk $\ell$, that is the loop obtained from $\ell$, by sequentially deleting consecutive pairs of edges, that are reverses of one another\footnote{It is not hard to check that $[\ell]$ is well defined, see \cite[Lemma 2.1]{Cha15}.}.

\subsubsection{Negative Deformation}
Define a negative deformation of a loop $\ell=aeb$ at the edge $e$ with the plaquette $p=ced$ by $\ell \ominus p= [ac^{-1}d^{-1}b]$. Notice that here the edge $e$ occurs with the same orientation in $\ell$ and $p$, if they occur with different orientations, we can still define negative deformation as follows. Define a negative deformation of a loop $\ell=aeb$ at the edge $e$ with the plaquette $p=ce^{-1}d,$ by $\ell \ominus p=[adcb]$. Observe that, negative deformation of $\ell$ with the plaquette $p$ or $p^{-1},$ gives the same loop. 
%However we shall only allow negative deformation with positively oriented plaquettes, so always exactly one of those will be allowed.
\begin{figure}[h]
\centering
\includegraphics[width=.7\textwidth]{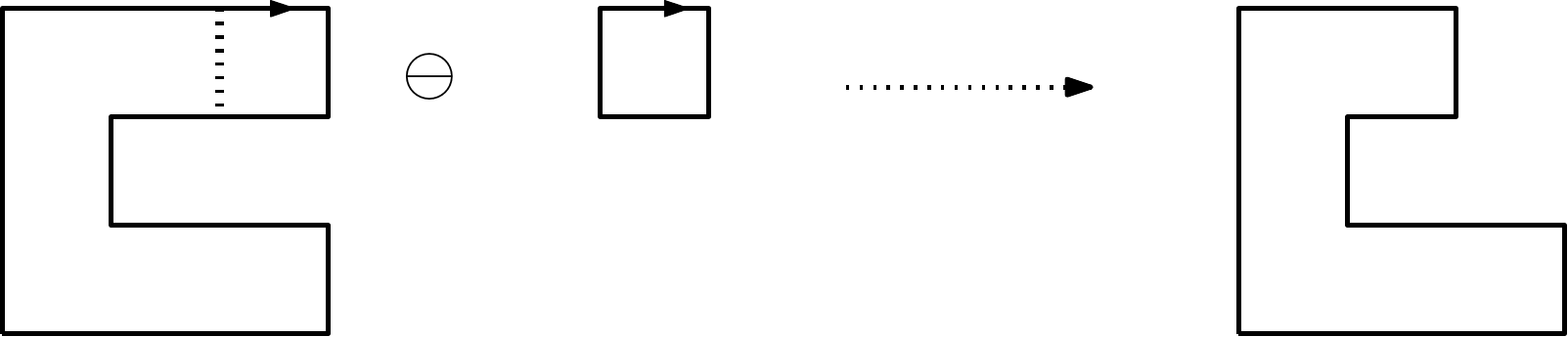}
\caption{Negative deformation}
\label{f:nd}
\end{figure}

\subsubsection{Positive Deformation}
Define a positive deformation of a loop $\ell=aeb$ at the edge $e$ with the plaquette $p=ced,$ by $\ell \oplus p= [aedceb]$. Positive deformation with different orientation, can be defined as before by taking $\ell=aeb$ and $p=ce^{-1}d,$ and defining $\ell\oplus p= [aec^{-1}d^{-1}eb]$.
Note that negative deformation along the edge $e,$ deletes the edge $e$, whereas for positive deformation along the edge $e$, the edge occurs with one extra multiplicity in the resulting loop.

Also, observe that the edges $e$ can occur at different locations along the loop $\ell$, so one needs to specify the locations $x$ along the loop $\ell$ (along with the edge $e$) to define the deformation operation. We shall denote those operations by $\ell \oplus _{x} p$ and $\ell\ominus _x p$ respectively. 
\begin{figure}[h]
\centering
\includegraphics[width=.7\textwidth]{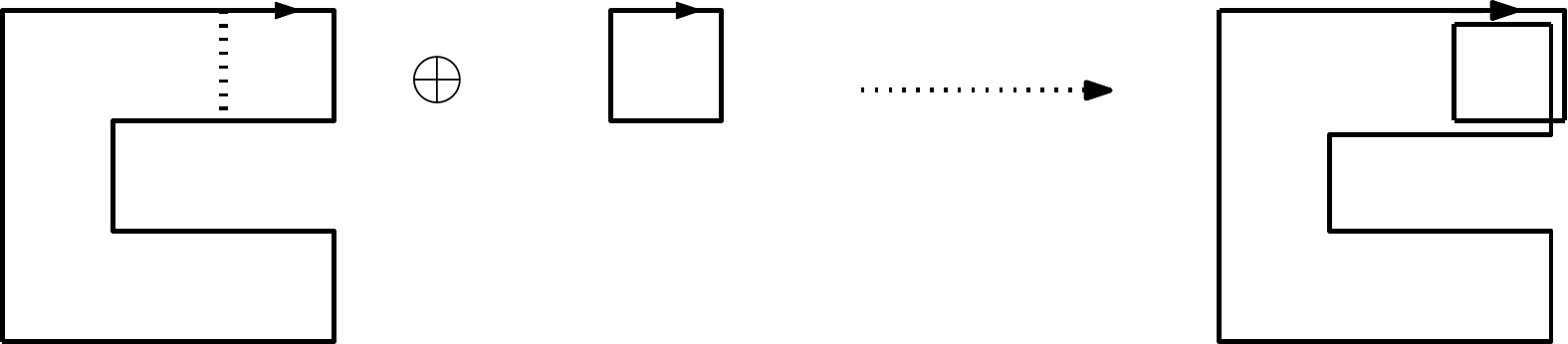}
\caption{Positive deformation}
\label{f:pd}
\end{figure}

\subsubsection{Splitting}
Splitting, as the name suggests, splits a single loop into two loops. There are two types of splitting, \textbf{positive splitting} and \textbf{negative splitting}. Positive splitting can occur when an edge is repeated twice in a loop with the same orientation. The positive splitting of the loop $aebec$ at the edge $e$ at the locations pointed to above (i.e.\ between $a$ and $b$ and between $b$ and $c$) is given by the pair of loops $[aec]$ and $[be]$. 

Negative splitting can occur when for an edge $e,$ both $e$ and $e^{-1}$ are present in the loop. Negatively splitting $aebe^{-1}c$ at the edges $e$ and $e^{-1}$ (at the specified locations as above) results into the pair of loops $[ac]$ and $[b]$. 

Notice that as with deformations, in case of splitting too, positive operations keep the edge, whereas negative operations delete it.
\begin{figure}[h]
\centering
\includegraphics[width=.7\textwidth]{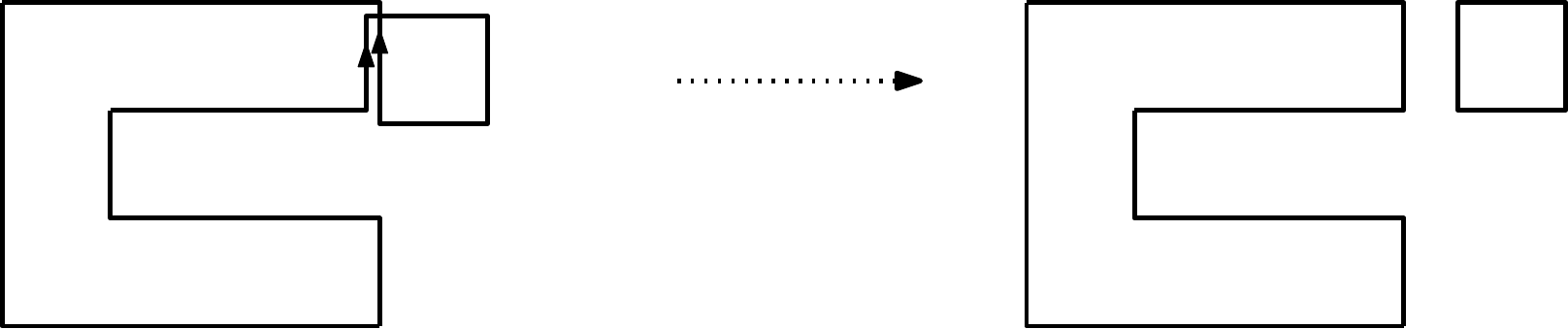}
\caption{Positive splitting}
\label{f:ps}
\end{figure}

Also observe that the edges $e$ and $e^{-1}$ can occur at different locations along the loop $\ell$, and hence  one needs to specify the locations $x$ and $y$ along the loop $\ell$ to define the splitting operation. We shall denote by $(\times_{x,y}^{1} \ell, \times_{x,y}^{2} \ell)$ the pair of loops created by splitting $\ell$ at locations $x$ and $y$ provided the operation is well-defined, i.e., the edges at the locations $x$ and $y$ are the same or are reversal of one another.  

\begin{figure}[h]
\centering
\includegraphics[width=.7\textwidth]{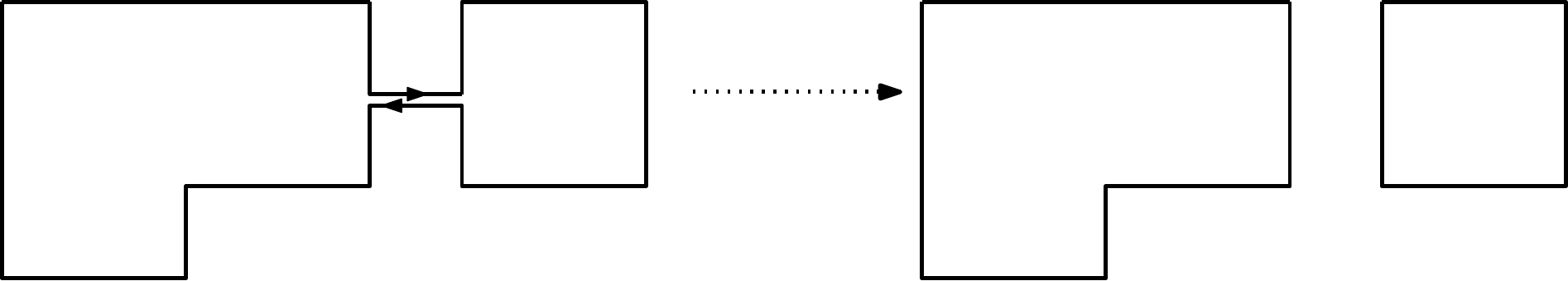}
\caption{Negative splitting}
\label{f:ns}
\end{figure}

\subsection{Master Loop Equation and the Fundamental Recursion}
\label{rec1}
A major tool used in Chatterjee's proof was the ``finite $N$ master loop equation" (see \cite[Theorem 3.6]{Cha15}) which is a set of recursive set of equations for the Wilson loop expectations. This type of equations has appeared in non-rigorous physics literature starting with the work of Makeenko and Migdal \cite{MM}, and similar results in the large $N$ limit for certain matrix models, have been obtained in \cite{CGM09} using the language of non-commutative derivatives. A major highlight of Chatterjee work is the application of  Stein's exchangeable pair to obtain the equations for finite $N$ whereas all previous results were obtained in the large $N$ limit. 

However, instead of the symmetric version of this recursion given in \cite[Theorem 3.6]{Cha15}, the un-symmetrized version of the recursion where one  focusses on all loop operations on a single edge $e$,  will be most useful for us. Let $\ell$ be a loop. Let $e$ be a fixed edge in $\ell$ that occurs with multiplicity $m$ (both $e$ and $e^{-1}$ together). Let $A$ (resp.\ $B$) be the set of locations in $\ell$ where $e$ (resp.\ $e^{-1}$) occurs. Let $C=A\cup B$. Let $\mathcal{P}^{+}(e)$ denote the set of all positively oriented plaquettes containing the edge $e$. Consider the recursive formula
\begin{eqnarray}\label{rec10}
a_{k}(\ell) &=& \frac{2}{m}\sum_{x\in A, y\in B} a_{k}(\times _{x,y}^{1} \ell, \times _{x,y}^{2} \ell) ~(\text{negative splitting})\\
\nonumber
&~& -\frac{1}{m} \sum_{x,y\in A, x\neq y} a_{k}(\times _{x,y}^{1} \ell, \times _{x,y}^{2} \ell) ~(\text{positive splitting with $e$})\\ 
\nonumber
&~& -\frac{1}{m} \sum_{x,y\in B, x\neq y} a_{k}(\times _{x,y}^{1} \ell, \times _{x,y}^{2} \ell) ~(\text{positive splitting with $e^{-1}$})\\
\nonumber
&~&+ \frac{1}{m}\sum_{x\in C}\sum_{p\in \mathcal{P}^{+}(e)} a_{k-1}(\ell\ominus _x p)  ~(\text{negative deformation with plaquette $p$ containing $e$})\\
\nonumber
&~&-\frac{1}{m}\sum_{x\in C}\sum_{p\in \mathcal{P}^{+}(e)} a_{k-1}(\ell\oplus _x p) ~(\text{positive deformation with plaquette $p$ containing $e$}).
\end{eqnarray}
\begin{remark}\label{prod}
The splitting terms above in the RHS are defined as follows. Recall that it follows from Theorem \ref{t:factor} that for a loops sequence $(\ell_1,\ell_2, \ldots, \ell_n)$ for limit of the loop expectation $\langle W_{\ell_1}W_{\ell_2}\cdots W_{\ell_n}\rangle $, scaled by $N^n$ also has a power series expansion as in \eqref{e:series}, denote the $k$-th coefficient is denoted by $a_{k}(\ell_1,\ell_2, \ldots, \ell_n)$. 
\end{remark}

For a loop sequence $(\ell_1, \ell_2, \ldots, \ell_n)$ and an edge $e$ one can define a similar recursive equation where the  the expansion about the edge $e$ is done for the first loop $\ell_1$ only. We shall call these recursions \textbf{the fundamental recursion}, and indeed it will be the fundamental tool used in our analysis.

Chatterjee proves that (see \cite[Proposition 4.1]{Cha15}) that the coefficients $a_{k}$ in \eqref{e:series} (as well as the multiple loop version of it) satisfy the fundamental recursion above, and the coefficients $a_{k}$ can be recursively computed using the above formula along with the initial condition that the power series for a null loop sequence is given by $1$ (i.e., $a_0=1$ and $a_{k}=0$ for all larger $k$). Note that it is not a priori clear that the recursion terminates as by doing positive or negative deformation, a loop can be made into a larger loop. However, as shown by Chatterjee, not only does the recursion terminate, but the solution is unique in the strongly coupled regime. The proof of the unicity as based on a contraction argument as in \cite{CGM09}, which in particular does not preclude the fact that even if the solution leads to a power series that has a larger radius of convergence, one can only show that the limiting value of the loop expectations is given by the power series only in the strong coupling regime on which Theorem \ref{t:looplimit} remains valid. This will indeed be one situation we shall encounter.

\subsection{Free probability and combinatorics}
The notion of `freeness' was introduced by Voiculescu around 1985 in connection with some old questions in the theory of operator algebras.  Furthermore, he advocated the point of view that freeness behaves in some respects like an analogue of the classical probabilistic concept of `independence' - but an analogue for non-commutative random variables, in particular for certain algebras generated by random matrices.  It turns out to be the right notion to analyze limiting loop expectations in the setting of planar lattice gauge theory. The key fact about the gauge theory Hamiltonian is a certain invariance property under conjugation by elements of $O(N)$ that allows us to  decouple the Gibbs measure and show asymptotic freeness. This is done in details in Section \ref{fa}.

Our work begins with the observation that one can make judicious choices of edges in using the fundamental recursion so that the solution becomes combinatorially tractable. Indeed our proof of Theorem \ref{t:plaquette} solely relies on this combinatorics without using any analytical machinery or free probability techniques. It might be possible to write down a proof of Theorem \ref{t:simple} in the language of this combinatorics as well, however using the connection to freeness provides insights on connections with other well known combinatorial objects such as non-crossing partitions. In particular by way of proving Theorem \ref{t:simple} we also prove that disjoint loops are asymptotically free, a fact that might be of independent interest; see Proposition \ref{free120} for a precise statement. Unfortunately, as one might expect, the connection to free probability is lost in higher dimensions and the combinatorics becomes more complicated. However, fortunately, a correspondence between string trajectories and non-crossing partitions still continues to persist in any dimension which we exploit to show analyze loop expectations.
% and show that  a general  \textcolor{red}{counterexample} Wilson loop area law lower bound, (see Proposition \ref{p:lbcounterex}). 

We also wish to emphasize that other approaches towards understanding  other lattice gauge theories in two dimensions (\cite{GW80, Wadia, LDPunitary})  are extremely reliant on the planar nature of the problem. The only promising approach in high dimensions seems to be through understanding  geometrical properties of random surfaces formed by the string trajectories akin to the decorated trees we encounter in the planar case (see Figure \ref{f:tree}), through analyzing the fundamental recursions.

%\textcolor{red}{organization needs to be rewritten}
\subsection{Organization of the article}
The rest of this article is organised as follows. In Section \ref{pla} we study in detail the limiting statistics for a single plaquette and prove Theorem \ref{t:plaquette}. In Section \ref{gf} we take advantage of the planarity of the setting which forces a lot of decoupling in the Gibbs measure. All of this is formalized using what is known in the physics literature as gauge fixing. In Section \ref{fa} we show joint convergence of all the plaquettes to an unital algebra of non-commutative variables and a linear functional. To do this, among other tricks  certain basic results of free probability theory are employed. These are reviewed in Section \ref{rev}. The entire proof of freeness crucially depends on the fact that for a single plaquette this convergence holds. Using this and certain tricks of free probability theory in Section \ref{s:poly} we give a description of loop statistics for the planar lattice gauge theory and prove Theorem \ref{t:poly}. A different gauge fixing allows us to prove that disjoint loops are asymptotically free and conclude the proof of Theorem \ref{t:simple}. This is done in Section \ref{s:loopfree}. In Section \ref{s:comp}, we do some explicit computations for loop statistics for some non-simple loops using free probability techniques, and in particular prove Proposition \ref{p:lbcounterex} for the planar case. In Section \ref{s:hd} we generalize this and discuss examples for which the area law lower bounds do not hold.  We finish with a discussion of some intriguing open questions in Section \ref{s:oq}.

\section{Statistics for a Plaquette}\label{pla}
We shall prove Theorem \ref{t:plaquette} in this section. Recall the Gibbs measure $\mu_{\Lambda_N,N,\beta}$. In what follows $Q=\{Q_{e}\}$ shall be a configuration of matrices drawn from $\mu_{\Lambda_{N}, N, \beta},$ where we shall suppress $N$ from the notations for convenience. As before, for any positively oriented plaquette $p$, we shall denote by $Q_{p},$ the product of the matrices along the edges of $p$; and for any loop $\ell$, we shall denote by $W_{\ell}$ the trace of the matrix obtained by multiplying the $Q$-matrices along the edges of $\ell$. For the purpose of this section $\langle \cdot \rangle$ will denote the expectation with respect to the Gibbs measure $\mu_{\Lambda_{N},N,\beta}$ where the parameters will always be clear from the context. The following theorem characterizes the loop expectations where the loop is either a plaquette, or a plaquette wrapped around multiple times.  

\begin{thm} 
\label{pla100} 
For any $k\in \N$ and any plaquette $p$,
$$\lim_{N\to \infty}\frac{\langle \Tr (Q_p^{k}) \rangle}{N} = \left \{ \begin{array}{cc}
\beta &  k=1\\
0 &  \text{ otherwise.}
 \end{array}
 \right.
 $$
\end{thm}
Observe that Theorem \ref{t:plaquette} is a special case $k=1$ of the above result. Also observe that $W_{\ell}=\Tr(Q_p^{k})$ where $\ell$ is the loop obtained by wrapping the plaquette $p$ around $k$ times; see Figure \ref{f:wrap}.
\begin{figure}[h]
\centering
\includegraphics[width=.12\textwidth]{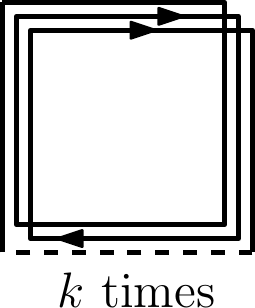}
\caption{A loop obtained by wrapping around a plaquette $k$ times.}
\label{f:wrap}
\end{figure}

A interesting question related to the above is the identification of joint spectral distribution of  the plaquette variables. Via the moment-method, for a single plaquette $p$ in dimension $2$, Theorem \ref{pla100} characterizes the limiting empirical spectral measure for the plaquette variable $Q_{p},$ as pointed out to us by Sourav Chatterjee and an anonymous referee. 
Moreover, one can exactly identify the limiting measure (supported on the unit circle) in this case. 
\begin{ppn}
\label{p:lsd}
Consider $SO(N)$ lattice gauge theory on $\Lambda_{N}\uparrow \Z^2$ as $N\to \infty$. Let $\beta$ be sufficiently small so that the conclusion of Theorem \ref{t:looplimit} holds. For a plaquette $p$, let $F_{N}$ denote the empirical spectral measure of the plaquette variable $Q_{p}$. As $N\to \infty$, $F_{N}$ converges weakly in probability to a deterministic measure $F$ on $S^1$ which has the following density (under the standard parametrization of the unit circle):
$$ f(\theta)=\frac{1}{2\pi} (1+2\beta \cos \theta); \quad \theta\in [0,2\pi).$$
\end{ppn}
%\textcolor{red}{make sure the constant is correct}
Note that for $\beta=0,$ we recover the uniform measure on the circle which is well known to be the limiting spectral distribution of Haar distributed matrices in $SO(N)$ (see \cite{diaconis}).
We sketch a proof of Proposition \ref{p:lsd} and mention other related results in Section \ref{s:oq}.

The proof of Theorem \ref{pla100} will be broken into several smaller results to exhibit the various ways one can exploit the fundamental recursion in Section \ref{rec1}. Some of the proofs might be simplified using the machinery of free probability which we shall establish later. However in this section we will refrain from doing so as the techniques presented in this section has the potential of being applicable outside the realm of exact solvability, i.e.,\ in higher dimensions where the tools of free probability no longer apply. 
\begin{figure}[h]
\centering
\includegraphics[width=.14\textwidth]{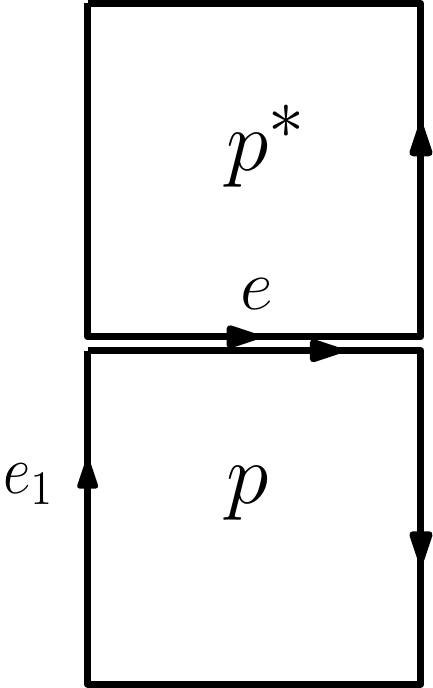}
\caption{Adjacent plaquettes sharing the edge $e$.}
\label{f:pp*}
\end{figure}

Recall the coefficients $a_k(\cdot)$ from \eqref{e:series}.
The goal of this part is to compute $a_k(p)$ for any plaquette  $p$.
Without loss of generality throughout this section we will assume that $p$ is the clockwise oriented plaquette whose bottom left corner is at the origin. It is easy to observe that $a_1(p)=1$. Thus the following proposition shall establish the $k=1$ case of Theorem \ref{pla100}.

\begin{ppn}\label{pla1}For all $k>1,$
$a_{k}(p)=0.$
\end{ppn}

Recall the fundamental recursion in Section \ref{rec1}. For a fixed edge $e$ in $p$, it gives a recursive expression for $a_{k}(p)$ in terms of the $k$-th or lower coefficients of loops or loop sequences obtained from $e$ by elementary operations  (splitting and deformation of the positive or negative kind). Those coefficients can again be recursively evaluated by using the fundamental recursion on them choosing different edges $e$ (we shall say that we use the fundamental recursion \textbf{rooted at $e$} in such a situation). We shall prove Proposition \ref{pla1} and the remaining part of Theorem \ref{pla100} by applying the fundamental recursion sequentially rooted at a carefully chosen sequence of edges that make the solution tractable. Before starting with a systematic analysis of the recursive equations, we give a one-step illustration.   

Let $e$ be the topmost edge of the plaquette $p$. Indeed, we shall always apply the fundamental recursion rooted at the topmost horizontal edge of the loop sequence in question. Let $p^*$ be the plaquette right above $p$ (see Figure \ref{f:pp*}).
 The fundamental recursion rooted at $e$ gives,
\begin{equation}
\label{e:1step}
a_{k}(p)=a_{k-1}(\emptyset)+ a_{k-1}(p\ominus p^*)-a_{k-1}(p\oplus p)- a_{k-1}(p\oplus p^*). 
\end{equation}

Notice that $p$ does not have any repeated edge, and hence the splitting terms are not present in the recursion. Now for $k>1$,  clearly $a_{k-1}(\emptyset)=0$. Each of the other terms now need to be evaluated recursively using the fundamental recursion. For small values of $k$, one can do these computations by hand but to prove the general result we need to study the recursions systematically. To do this we shall parametrize the loops by certain trees where every generation is decorated with a spin in $\{\pm1\},$  as shown in Figure \ref{f:tree}. 

\begin{figure}[h]
\includegraphics[width=.65\textwidth]{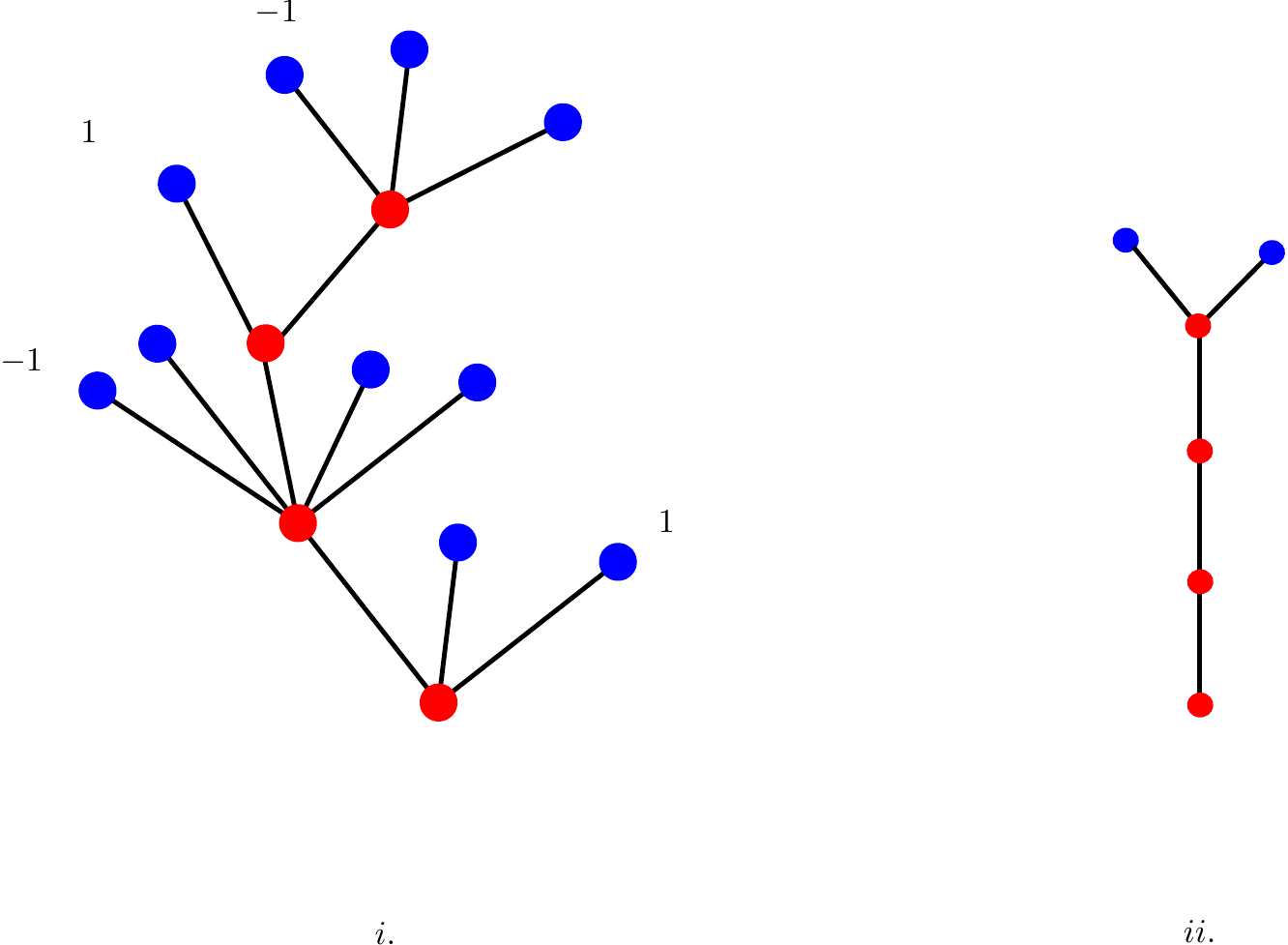}
\caption{$i.$ Decorated tree representation of trajectories. All the elements of a generation have the same $\pm1$ spin attached to them. Moreover in each generation, exactly one element (the red coloured) vertex has offspring. $ii.$ The tree where every generation has size one except the last generation which has size two. (appears in the proof of Lemma \ref{aux1}.) One can think of the vertices in each level of the trees as ordered naturally from left to right in the above representation.}
\label{f:tree}
\end{figure}

Formally we do the following. Let $\mathfrak{T}$ denote the space of all rooted trees, where at each level except at most one, all other vertices are leaves; vertices at each level are ordered and a spin in $\{\pm 1\}$ is associated to each of the levels. We encode a rooted tree $\T$ in $\mathfrak{T}$ with $h\geq 0$ levels in the following way: when $h=0,$ the root is the only vertex at level $0$. If $h\ge1,$ the encoding is done using the sequence  $\{(g_{i},\, k_i,\, \sigma_i),\, i=1,2,\ldots, h\},$ where $g_i$ is the number of vertices at level $i$; we use $k_{i}\leq g_{i}$  to denote the index of the vertex at level $i$ (according to the order of the vertices) which is not a leaf, i.e., which has offspring in $(i+1)-th$ generation. For $i=h$, formally we denote $k_i=0$, because all the vertices at level $h$ are leaves. The spin associated with the level $i$ of $\T$ is denoted by $\sigma_{i}\in \{\pm 1\}$.

Given any tree $\T$ in $\mathfrak{T}$, we now associate a unique loop $\ell(\T)$ to it. For $h=0,$ the single vertex root corresponds to the null loop. Otherwise, $\ell(\T)$ will always be a loop whose bottom left corner is the origin, i.e., the same as bottom left corner of $p$. We shall always describe the loop starting from its bottom left corner (recall a loop is a cyclically equivalent sequence of edges that starts and ends at the the same point). Also for a tree $\T$ and $m\in \Z$, we shall denote by $\ell (\T)_{m}$, the loop obtained by translating each edge of the loop $\ell (\T)$ by $m$ units vertically upwards.

We are now ready to describe the recursive construction of the loop $\ell (\T)$. For the base case $h=1$, if $\T=\{(g_1,0,\sigma_1)\}$, then $\ell(\T)=p_{1,\sigma_1}p_{1,\sigma_1}\cdots p_{1,\sigma_1}$ ($g_1$ times), where $p_{1,1}=p$ and $p_{1,-1}=p^{-1}$. That is, in this case $\ell(\T)$ is just the plaquette $p$ wrapped around $g_1$ times with positive or negative orientation depending on $\sigma_1$. Now let $e_1$  denote the left edge of $p$ oriented upwards (see Figures \ref{f:pp*} and \ref{illustration1}). Suppose we have defined the loop $\ell(\T)$ for all trees $\T$ with less than $h$ levels, and let us now consider $\T=\{(g_i,k_i,\sigma_i):i=1,2,\ldots, h\}$ with $h>1$ which clearly implies that $k_1\neq 0$. Let $\T'$ be the tree obtained from $\T$ by deleting the first level, i.e.,
$$\T'=\{(g'_i, k'_i,\sigma'_i):i=1,2,\ldots, h-1\},$$
where $g'_i=g_{i+1}$, $k'_i=k_{i+1}$ and $\sigma'_i=\sigma_{i+1}$. We define $\ell(\T)$ by the following recursive rule: 
\begin{itemize}
\item if $\sigma_1=1,$
$$\ell(\T)=(p_{1,\sigma_1} \underset{k_1-1 \text{ times}}{\cdots} p_{1,\sigma_1}) e_1 \ell(\T')_{1} e_1^{-1} (p_{1,\sigma_1} \underset{g_1-k_1+1 \text{ times}}{\cdots} p_{1,\sigma_1}).$$ 
\item if $\sigma_1=-1,$
$$\ell(\T)=(p_{1,\sigma_1} \underset{k_1 \text{times}}{\cdots} p_{1,\sigma_1}) e_1 \ell(\T')_{1} e_1^{-1} (p_{1,\sigma_1} \underset{g_1-k_1 \text{times}}{\cdots} p_{1,\sigma_1}).$$ 
\end{itemize}

Let us take a moment to parse the above definition. In the first case we start by wrapping the plaquette $p_{1,1}$ around $k_1-1$ times. By our convention this starts and ends at the origin. Then we move up along the edge $e_1$. Then we trace out the recursively defined loop $\ell(\T')_1$ (observe that after $\ell(\T')$ is translated upwards by one unit, it's starting (and ending) point coincides with the ending point of $e_1$). Finally we trace the edge $e_1^{-1}$ (moving vertically downward to the origin) and finish by wrapping the plaquette $p_{1,\sigma_1}$ around again $g_1-k_1+1$ times.
Note that when $\sigma_1=-1,$ we wrap around initially $k_1$ times instead before tracing $\ell(\T')_1$. The reason behind this, will be clear from the proof of the following lemma (also see Figures \ref{illustration1} and \ref{illustration2}).  For the next lemma, for any $\T \in \mathfrak{T},$ let $e_{\T}$ be the topmost horizontal edge appearing in $\ell(\T)$. The 
next lemma shows that the space of loops corresponding to trees is closed under the operations deformation and splitting at $e_{\T}.$

\begin{lem}\label{closed}For any $\T \in \mathfrak{T},$ all the loops obtained from $\ell(\T),$ by performing deformations or splitting, rooted at $e_{\T},$  belong to the set $\{\ell(\T')_m:\T' \in \mathfrak{T}, m \in \Z\}.$
\end{lem}

%\textcolor{red}{something about the convention for the null loop?}

\begin{figure}[h]
\includegraphics[width=.85\textwidth]{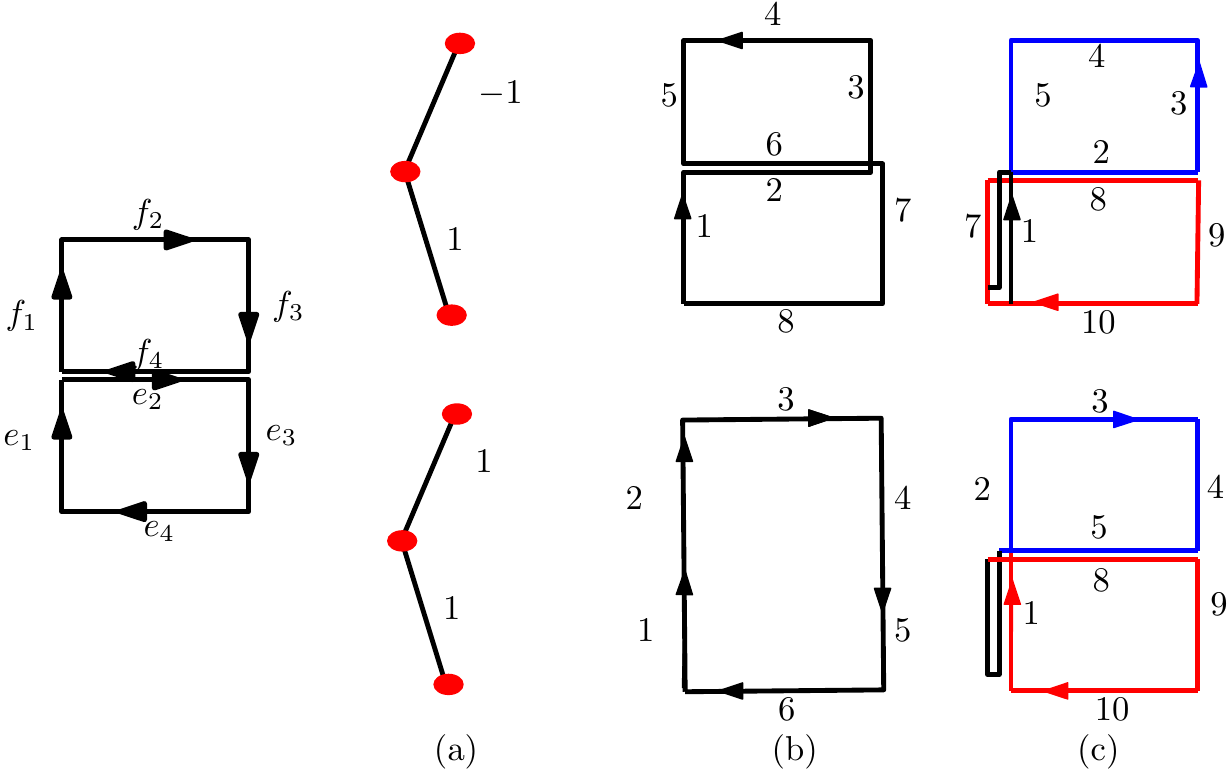}
\caption{The above figure compares the tree construction when a deformation is applied on the plaquette $p_{1,1}$ rooted at $e_2$. The top and the bottom figures correspond to the loops $p_{1,1}\oplus p_{2,-1}$ and $p_{1,1}\ominus p_{2,1},$ respectively. (a) shows the trees, (b) shows the loops obtained, (c) describes the loops obtained by the rules of constructing $\ell(\T).$ }
\label{illustration1}
\end{figure}
%Before proceeding we need a few notations. For any $\T$, with heigh $h$ let $p_{\T}=\tilde e_1\tilde e_2\tilde e_3\tilde e_4$ be the plaquette adjacent to but below $e_{\cT}$. Now if $h$ is the height of the tree,  then 

\begin{proof} We will mention the main observations leading to the proof, often skipping the clear but tedious combinatorial details. 
Fix $\T\in \mathfrak{T},$ and let $h$ be the height of the tree and $p_{h,1}$ denote the clockwise oriented plaquette  starting at the origin, shifted up at height $h$ and similarly  $p_{h,-1}$ denotes the counterclockwise version. 
We first verify that the lemma is true for $\T$ corresponding to the plaquettes $p_{1,\pm1}$ i.e., $(1,0,\pm 1)$. For concreteness let $p_{1,1}=e_1e_2e_3e_4$ and $p_{2,1}=f_1f_2f_3f_4,$ (see Figure \ref{illustration1}).
Now the case of deforming $p_{1,1}$ using $p_{1,\pm1}$ 
 is easy, and hence we consider the other cases.
Note that from definitions (Section \ref{cst}),  
\begin{align*}
p_{1,1}&\oplus p_{2,-1} =e_1f^{-1}_4f^{-1}_3f^{-1}_2f^{-1}_1e_2e_3e_4,\\
p_{1,1}&\ominus p_{2,1} \,\,\,=e_1f_1f_2f_3e_3e_4.
\end{align*}
%\textcolor{red}{the first one above is clearly wrong. should it be $f_4^{-1}f_3^{-1}??$, the first one below is wrong too, same mistake?}
Now let $\T_1=\{(1,1,1),(1,0, -1)\}$ and $\T_2=\{(1,1,1),(1,0,1)\}.$
%Whereas our definitions yield trees (\textcolor{red}{what trees?}) $\T_1$ and $\T_2$ in the above cases such that
Then according to our definitions,
\begin{align*}
\ell(\T_1) &=e_1\underbrace{f^{-1}_4f^{-1}_3f^{-1}_2f^{-1}_1}_{p_{2,-1}}e^{-1}_1e_1e_2e_3e_4,\\
\ell(\T_2) &=e_1\underbrace{f_1f_2f_3f_4}_{p_{2,1}}e^{-1}_1e_1e_2e_3e_4.
\end{align*}
Similarly if we started with $p_{1,-1}$ instead, then we have the following, 
\begin{align*}
p_{1,-1}&\oplus p_{2,1}\,\,\, =e^{-1}_4e^{-1}_3f_4f_1f_2f_3e^{-1}_2e^{-1}_1,\\
p_{1,-1}&\ominus p_{2,-1}=e^{-1}_4e^{-1}_3f^{-1}_3f^{-1}_2f^{-1}_1e^{-1}_1.
\end{align*}
In this case let, $\T_1=\{(1,1,-1),(1,0, 1)\}$ and $\T_2=\{(1,1,-1),(1,0,-1)\}.$
Once again our definitions yield,
\begin{align*}
\ell(\T_1) &=e^{-1}_4e^{-1}_3e^{-1}_2e^{-1}_1e_1\underbrace{f_1f_2f_3f_4}_{p_{2,1}}e^{-1}_1, \\
\ell(\T_2) &=e^{-1}_4e^{-1}_3e^{-1}_2e^{-1}_1e_1\underbrace{f^{-1}_4f^{-1}_3f^{-1}_2f^{-1}_1}_{p_{2,-1}}e^{-1}_1.
\end{align*}
It is easy to check that the actual definitions and the tree constructions are equivalent up to backtracks.  
The proof now can be completed by observing that in each of the tree constructions,   a copy of the plaquette $p_{2,\pm1}$ exists.
Note that this is ensured by the two different definitions for different orientations, of the tree to loop map .  Now loop operations on the trees $\T_1$ and $\T_2,$ are rooted at $f_2$ (see Figures \ref{illustration1} and \ref{illustration2}).  Thus we are done using induction and the previous argument repeatedly along with the observation that  any splitting at the top edge, only results in another component which is a wrap around of a plaquette which clearly corresponds to a tree as well.  %\red{improve wording, last few lines make no sense, plus it is not stated why it was important to have two different definitions}
\end{proof}

\begin{figure}[h]
\includegraphics[width=.85\textwidth]{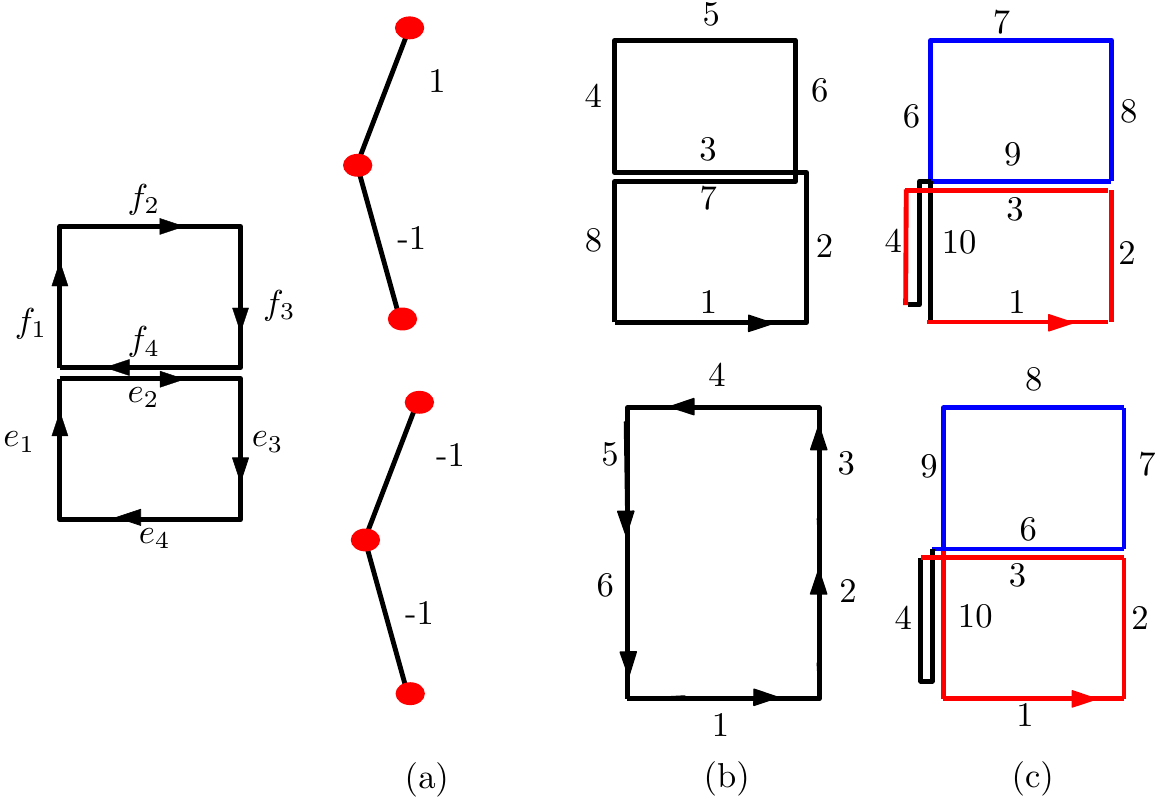}
%\caption{Explaining tree to loop bijection..\red{add more words}}
\caption{As in Figure \ref{illustration1}, the above figure compares the tree construction when a deformation is applied on the plaquette $p_{1,-1}$. The top and the bottom figures correspond to the loops $p_{1,-1}\oplus p_{2,1}$ and $p_{1,-1}\ominus p_{2,-1},$ respectively.(a) shows the trees, (b) shows the loops obtained, (c) describes the loops obtained by the rules of constructing $\ell(\T).$  The numbers on the edges reveal the order in which they are traversed. }
\label{illustration2}
\end{figure}

Notice that for $\ell(\T)$ as above, the vertical edges at every level $i$ occurs at least $g_i$ many times (always with the same orientation) and hence any $2$-chain $x$ whose boundary is $\ell(\T),$ {must contain the plaquette $p_{i}$ (or its inverse) with coefficient at least $g_i$. From \eqref{area} it follows that}
for $\T=\{(g_i,k_i,\sigma_i):i=1,2,\ldots , h\},$ we have
\begin{equation}\label{obs1}
{\rm{area}}(\ell(\T))=\sum_{i=1}^{h}{g_{i}}.
\end{equation}

Recall that our objective is to evaluate the coefficients $a_{k}(p)$. The reason behind introducing the trees and the associated loops is that while we repeatedly use the fundamental recursion rooted at a  carefully chosen edge $e,$  the loops we shall obtain shall all be associated with trees in $\mathfrak{T}$ in the manner described above. Thus our analysis involves computing  $a_{k}(\ell(\T))$ for $\T\in \mathfrak{T}$. We now quote the  following result established by Chatterjee \cite[Lemma 14.1]{Cha15} which we will use. 
\begin{lem}\label{min1} For any non-null loop $\gamma$,  the minimum number of deformations in a vanishing trajectory starting from $\gamma$ is at least $\rm{area}(\gamma).$
\end{lem}
In the above lemma, a vanishing trajectory is a sequence of loop(s) starting with $\gamma$ and finishing with the null loop, where each element of the sequence is obtained from the preceding one from a deformation or splitting operation. Theorem 3.1 of \cite{Cha15} implies that the if any vanishing trajectory starting from $\gamma$ requires at least $k$ deformations, then $a_j(\gamma)=0$ for all $j<k$. In particular, Lemma \ref{min1} implies that for any loop $\ell$ we have  $a_{k}(\ell)=0$ for all $k< {\rm area}(\ell)$. We shall also need the following easy consequence of Lemma \ref{min1} and Theorem \ref{t:factor}.

\begin{lem}
\label{l:coefffactor}
For loops $\ell_1$ and $\ell_2$ with areas $k_1$ and $k_2$ respectively we have 
$$a_{k_1+k_2}(\ell_1,\ell_2)= a_{k_1}(\ell_1)a_{k_2}(\ell_2).$$
\end{lem}

With these auxiliary results at our disposal, we can now prove the following comparison of coefficients between two loops $\ell(\T)$ and $\ell (\T')$.

\begin{lem}\label{lem2}
For any two trees $\T=\{(g^{(1)}_i,k^{(1)}_i,\sigma^{(1)}_i): i=1,2,\ldots, h\}$ and $\T'=\{(g^{(2)}_i,k^{(2)}_i,\sigma^{(2)}_i): i=1,2,\ldots, h\}$ with the property that for all $1\le i\le h,$
$$(g^{(1)}_i,k^{(1)}_i)=(g^{(2)}_i,k^{(2)}_i),$$ for $k={\rm{area}}(\ell(\T))={\rm{area}}(\ell(\T')),$
\begin{equation}\label{area1}
a_{k}(\ell(\T))=a_{k}(\ell(\T')).
\end{equation}
 \end{lem}
\begin{proof}
For brevity, throughout this proof, we will denote $\ell(\T)$ and $\ell(\T')$ by $\ell$ and $\ell'$ respectively.
The proof is by induction on the area of $\ell$ and the recursion in \eqref{rec10}. For any $\T$ as in the hypothesis, the edge on which we root the recursion is $e=e_{\T}$; the highest horizontal edge appearing  in $\ell$.  Note that by hypothesis this implies the highest horizontal edge $e'=e_{\T'}$ appearing in $\ell'$, is either $e$ or $e^{-1}$ depending on whether $\sigma^{(1)}_h= \sigma^{(2)}_h$ or $\sigma^{(1)}_h= -\sigma^{(2)}_h$ respectively. Further observe that the multiplicity of $e$ in $\ell$ is equal to the multiplicity of $e'$ in $\ell'$. We want to compare the terms on the right hand sides for the fundamental recursion for  $a_{k}(\ell)$ rooted at $e$ and the fundamental recursion of $a_{k}(\ell')$ rooted at $e'$.

Observe that $e$ in $\ell,$ (similarly $e'$ in $\ell'$) appears only in one orientation. So the negative splitting terms do not appear. Notice further that positive deformation with either of the plaquettes adjacent to $e$ (resp.\ $e'$) always increases the area by one and hence the terms corresponding to positive deformation is zero by Lemma  \ref{min1}. So we only need to show that the negative deformation terms and the positive splitting terms match up in both the recursions. 
To see this, notice that for negative deformation of $\ell$ with a plaquette $p_{e}$ at any location $x$ where the edge $e$ occurs in $\ell$, there exists a unique corresponding location $x'$ at $\ell'$ (where $e'$ occurs) and a plaquette $p_{e'},$ such that either both the negative deformations decrease the area or both increase the area. The cases where the area increases can be ignored as above, and for the cases where the area decreases, by Lemma \ref{closed}, we get two loops $\ell_*$ and $\ell'_{*}$ of strictly smaller area which still satisfy the hypothesis of the lemma;  i.e., there exists trees $\T_{*}$ and $\T'_{*}$ satisfying the hypothesis in the lemma such that $\ell_{*}=\ell(\T_{*})$ and $\ell'_*=\ell(\T'_{*}).$  By induction, these two terms are equal, and hence the negative deformation terms overall are equal for both the recursions.

The positive splitting terms can be dealt with similarly. 
Observe that whenever we perform a splitting along two appearances of the edge $e$ in $\ell$, we can find corresponding appearances of $e'$ in $\ell'$ such that the (translates of) resulting loops correspond to trees $(\T_1, \T_2)$ and $(\T'_1, \T'_2)$ respectively where the trees have the following properties:
\begin{itemize}
\item $\T_1$ and $\T'_1$ (also $\T_2$ and $\T'_2$) satisfy the hypothesis of the lemma. In particular, 
$${\rm area}(\ell (\T_1))= {\rm area}(\ell (\T'_1))=k_1~\text{and}~{\rm area}(\ell (\T_2))= {\rm area}(\ell (\T'_2))=k_2.$$
\item $k=k_1+k_2$.
\end{itemize}
That the splitting terms are equal now follows from the induction hypothesis and Lemma \ref{l:coefffactor}. This completes the proof of the lemma.
\end{proof}

We can now characterize all the trees $\T$ such that $a_k(\ell (\T))\neq 0$ where $k={\rm{area}}(\ell(\T)).$

\begin{lem}\label{aux1}For any $\T\in \mathfrak{T},$ and ${k}=\rm{area}(\ell(\T)),$
$$a_{k}(\ell(\T))=1\, \text{or}\,\, 0,$$ depending on whether $\T$ is a path or not.
\end{lem}
\begin{proof}
Suppose $\T=\{(g_i,k_i,\sigma_i): i=1,2,\ldots,h \}$ is a path, i.e., $g_i=k_i=1$ for all $1\le i<h$. We apply the fundamental recursion for $a_{k}(\ell(\T)),$ rooted at the topmost horizontal edge $e=e_{\T}$. Since the edge $e$ in $\ell(\T)$ has multiplicity one, it follows that the splitting terms do not contribute. As positive deformation increases the area,  the positive deformation term can be ignored as well. It only remains to deal with the negative deformation term. One can check that out of the two negative deformations, one (the one with the plaquette above $e$) increases area and hence can be ignored. For the other negative deformation, one gets the loop $\ell(\T')$ of area $k-1$ where  $\T'$ is the path obtained from $\T$ by deleting the leaf at the topmost level. It follows that $a_{k}(\ell(\T))=a_{k-1}(\ell(\T'))$ and we are done by induction (the base case $h=1,$ is trivial).

Now suppose $\T$ is not a path. Applying the fundamental recursion for $a_{k}(\ell(\T))$ rooted at the topmost horizontal edge $e$, as before, notice that only positive splitting and negative deformations are allowed. Any splitting, where one of the parts (as in the proof of Lemma \ref{lem2}) is not a path contributes zero by induction. Now suppose $\T$ has more than one vertex in any level except the topmost one, or has more than two vertices in the topmost level i.e., either $g_i>1$ for some $i<h$ or $g_h>2.$  In this case both negative deformation and splitting at $e$ leads to a loop of strictly smaller area, which corresponds to a tree that is not a path. At this point we are done by induction. 
Thus the  only remaining case is where $\T$ has one vertex in every level other than the topmost level and has two vertices at the topmost level; as in Figure \ref{f:tree} $ii.$.
%\textcolor{red}{figure placement}
The proof is now complete by noticing that in this case both the splitting term and the negative deformation term are $1$ (using the first part of this lemma) with opposite signs. 
\end{proof}
We are now ready to finish the proof of Proposition \ref{pla1}. We in fact prove the following stronger statement. 
Throughout the proof of the next lemma all the loops we will encounter will have a tree representation and hence we will use the two notions interchangeably as there would be no scope of confusion. 

For a loop sequence,  $s=(\T_1, \T_2,\ldots , \T_j)$, define ${\rm{area}}(s)=\sum_{i=1}^j{\rm{area}}(\T_i).$  
\begin{lem}
\label{aux2}
For any $s$ as above, $a_k(s)=0,$ whenever $k> {\rm{area}}(s).$ 
\end{lem}

Clearly taking $s$ to be the tree with a single edge (a single leaf connected to the root) implies Proposition \ref{pla1}.

\begin{proof}
The proof will follow from induction on $\alpha:=k-{\rm{area}}(s),$ (we suppress the dependence on $s$ in the notation as it will be clear from context). We shall establish the claim separately for $\alpha=1$ and $\alpha=2$. The argument for $\alpha=1$, which is essentially a parity argument generalizes easily for all odd $\alpha$; and the proof for general even $\alpha$ follows by induction.

\textbf{Case} $\mathbf{\alpha=1:}$
For convenience let ${\rm area}(s)=A$. The proof in this case goes along the following steps:

\begin{itemize}
\item By Lemma \ref{min1},  there are at least $A$ deformations needed to reach the null loop from $s$. 
\item By Lemma \ref{closed}, all the components formed by positive splitting (negative splitting never occurs) have a similar tree representation. Moreover, by \eqref{obs1}, whenever there is a positive splitting, the areas of the loops obtained by the positive splitting add up to the area of the loop which was split.

\item We apply the recursion in \eqref{rec10} repeatedly and this produces a virtual decision tree for the possible moves and at each node of this tree we have a loop sequence, (a geodesic/trajectory of this tree represents a path from a loop sequence ending with the null loop sequence.).  

\item Now consider any such trajectory starting from $s$ and finishing at the null loop. We shall show that this sequence contributes $0$ to $a_{A+1}(s)$ thereby proving this step. Observe also that for a trajectory to contribute to $a_{A+1}(s)$ it must contain exactly $A+1$ deformation steps.

\item  Consider for any loop sequence in such a trajectory, the possible deformations on any component loop/tree, and the effect it has on the area. By choice, at any point in time,  for any tree, we always use the recursion rooted at the topmost edge (say $e$). Let ${{p}_{+}}(e)$ and ${{p}_{-}}(e)$ be the two clockwise oriented plaquettes adjacent to $e$, (above and below $e$ respectively). 

It is easy to check that each of the four choices of positive and negative deformation with ${{p}_{+}}(e)$ and ${{p}_{-}}(e)$ changes the area by one. In particular,  only the negative deformation with ${{p}_{-}}(e),$ causes an area decrease by one and the other three increases the area by one.  

\item Now for any given trajectory, let $\tau$ be the first time, where an area increasing deformation is applied (there must be such a step, for the trajectory to contribute to $a_{A+1}(s)$). Let $A_1$ be the total area of the loop sequence at time $\tau-1$. Then $A-A_1$ is equal to the number of deformations (all of them are area decreasing) till then. Let  $(\T_1,\T_2,\ldots, \T_m)$ (see Remark \ref{prod}) be the loop sequence in this trajectory at time $\tau-1$. The  deformation at $\tau$ is applied on one of the component trees (say $\T_1$). At that point, the area  increases by one and hence the minimum number of deformations needed to reduce it to the null string also increases by one. Thus the total number of deformations needed to reduce the loop sequence to the null sequence, is at least $A_1+2$.  Hence the trajectory must contain at least $A+2$ deformation steps in total, which implies that it contributes $0$ to $a_{A+1}(s)$.
\end{itemize}

Essentially the same argument takes care of all odd values of $\alpha$, we omit the details. The next step is to treat the case $\alpha=2$.

\textbf{Case} $\mathbf{\alpha=2:}$ Lets us use the same notations as in the $\alpha=1$ case. For a vanishing trajectory of loop sequences starting from $s,$ let $\tau$ and $(\T_1,\T_2,\ldots, \T_m)$ be as before. We shall show that all trajectories which has the representation $(\T_1,\T_2,\ldots, \T_m)$ at time $\tau-1,$ combined, contribute $0$ to $a_{A+2}(s)$.  Since the next loop operation can be performed on any of the component trees we will analyze only the trajectories that perform the operation on $\T_1$ at time $\tau.$ This suffices since we are working with an arbitrary labeling of the trees.  As before, $e_{\T_1}$ is the top most edge of $\ell(\T_1)$, and let ${{p}_{+}}(e)$ and ${{p}_{-}}(e)$ denote  the two clockwise oriented plaquettes above and below $e$.) The proof will follow from the next two claims.
\begin{itemize}
\item $a_{{\rm area}(\T_1)+1}(\hat \T_1)=a_{{\rm area}(\T_1)+1}(\tilde \T_1)$
where $\hat \T_1$ and $\tilde \T_1$ are the trees obtained by positive  and negative deformations of the tree $\T_1$ with $ {{p}_{+}}(e)$ respectively. 
\item  $a_{{\rm area}(\T_1)+1}(\T'_1)=0$ where $\T'_1$  is the tree obtained by positive deformations of a tree $\T_1$ with $ {{p}_{-}}(e).$ 
\end{itemize} 

To see why these two claims suffice, note that  at time $\tau$ the possible loop sequences are $s_1=(\hat \T_1, \T_2,\ldots, \T_m)$, $s_2:=(\tilde \T_1, \T_2,\ldots, \T_m)$ and $s_3:=(\T'_1, \T_2,\ldots \T_m)$ all of which have the same area, (say $t$). 
Also up to $\tau$, the number of deformations made is equal to $k-\sum_{i=1}^{m}{\rm{area}}(\T_i)+1$ (the additional one is due to the positive deformation at $\tau$).  Since the number of deformations needed to reduce any of the above three sequences to the null string is at least $t,$ it follows from \eqref{rec10},  that the total contribution to $a_{A+2}$ is a  multiple of $a_{t}(s_1)-a_{t}(s_2)+a_{t}(s_3).$ 
Since except for the first component, all the other loops in $s_1,s_2$ and $s_3$ are exactly the same, the result follows from the above two claims. 

The first claim follows from Lemma \ref{lem2} after observing that $\hat \T_1$ and $\tilde \T_1$ satisfy the hypothesis of that lemma, where the orientation of the plaquettes only in the highest level occur with opposite signs. The second claim follows from Lemma \ref{aux1} and the observation that $\T_1'$ is not a path since it has at least two plaquettes at the highest level (because of the positive deformation). Thus the proof for $\alpha=2$ is complete.
\begin{figure}[h]
\centering
\includegraphics[width=.65\textwidth]{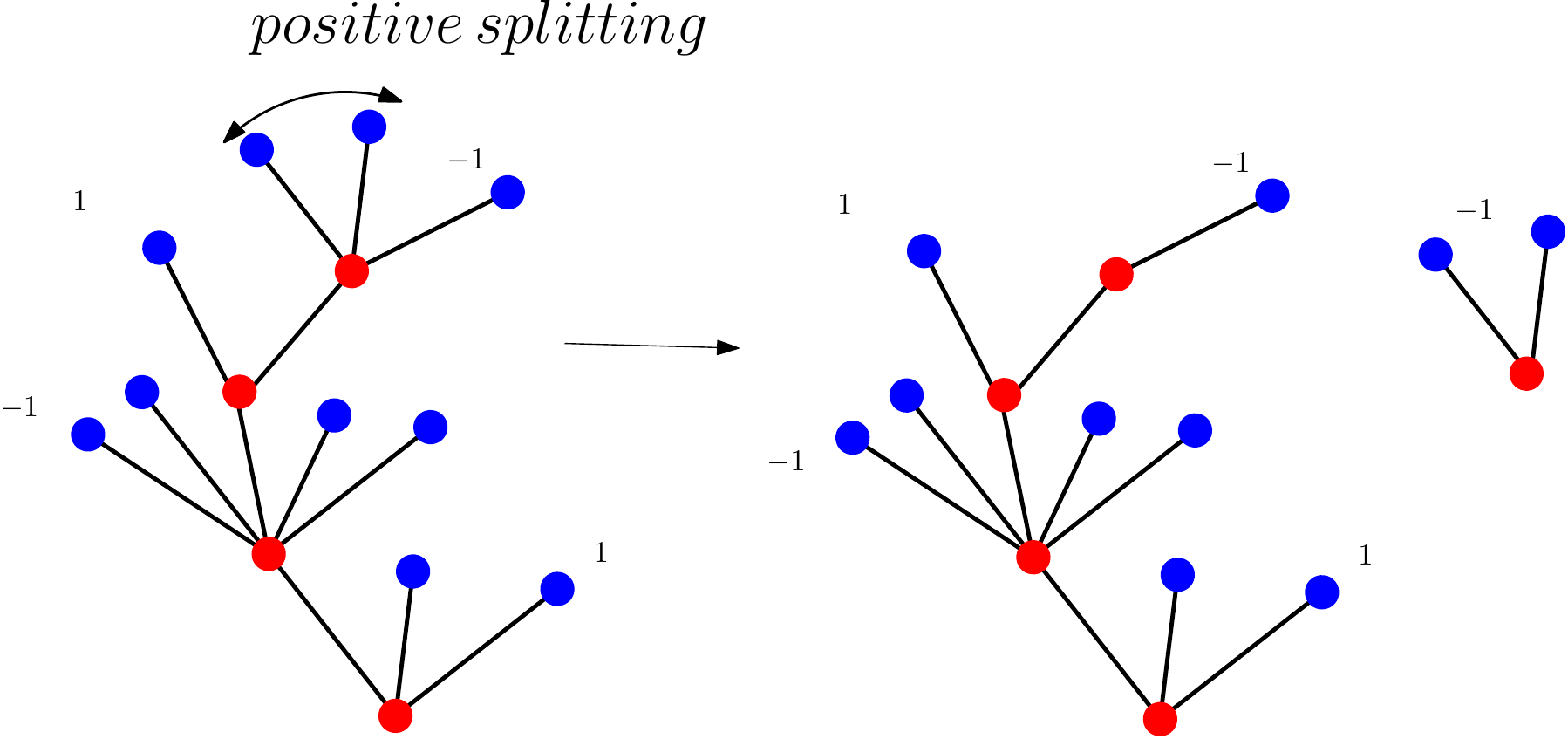}
\caption{Illustrates the formation of various trees by applying the fundamental recursion rooted at the top most edge.}
\label{string1}
\end{figure}

The proof for general $\alpha$ now follows by first applying the recursion repeatedly for any loop sequence till $\alpha$ decreases (an area increasing deformation is applied) for the first time. At that point $\alpha$ decreases by two and we are done by induction.

\end{proof}

We are now ready to finish the proof of Theorem \ref{pla100}.

\begin{proof}[Proof of Theorem \ref{pla100}]
The case $k=1$ has already been proved in Proposition \ref{pla1}. For $k>1$, let $\ell$ denote the loop obtained by wrapping the plaquette $p$ around $k$ times. Notice that $\ell=\ell(\T)$ where $\T=(k,0,1)$. Since $\T$ is not a path, it follows from Lemma \ref{aux1} and Lemma \ref{aux2} that the all the coefficients of $\ell$ are $0$ and hence $w(\ell,\beta)=0.$ 
\end{proof}

 \section{Gauge Fixing in two dimensions}\label{gf}
Often in various Hamiltonians like the one in \eqref{e:gibbs}, one has the freedom of forcing the value on certain edges. We will use it to great advantage in the planar setting. Although we shall use different variants of this technique (called gauge fixing), we start by describing a more classical variant, known as \emph{axial gauge fixing}, we will force the matrices on the vertical edges to be identity. The set of vertical edges in $E_{\Lambda} $ will be denoted by $V_{\Lambda}$  and similarly the corresponding horizontal edges will be denoted by $H_{\Lambda}$ \footnote{All our edge sets unless specifically oriented, should be thought of as containing two copies of each edge oriented in the two possible directions.}. 

Throughout this section we shall work with $\Lambda:=[-n,n]^2$ for some $n$. In this case, we let $V_{n}:= V_{\Lambda}$ (resp.\ $H_n:=H_{\Lambda}$). For any matrix $A$,  $\Tr(A)$ and $\tr(A)$ will denote the usual and normalized traces respectively. Let $O(N)$ denote the orthogonal group of order $N$, that is the group of all $N\times N$ orthogonal matrices with real entries. Recall that $\cQ$ is the space of all configurations. Let $G(\Lambda)$ denote the set of all maps from the vertices in $\Lambda$ to $O(N)$. Given $G\in G(\Lambda),$ it acts on $\cQ$ in the following natural way: for any $Q\in \cQ$,  for all neighbouring $x,y \in \Lambda,$ let us denote $Q_{e}$ by $Q(x,y)$ where $e$ is the directed edge which starts at $x$ and ends $y$ and define $GQ\in \cQ$ by
$$GQ (x,y)=G(x)Q(x,y)G^{-1}(y).$$
Note that since $SO(N)$ is a normal subgroup of $O(N),$ $GQ (x,y) \in SO(N)$ for all $(x,y).$ 
%\red{inconsistency}
\begin{figure}[h]
\centering
\includegraphics[width=.9\textwidth]{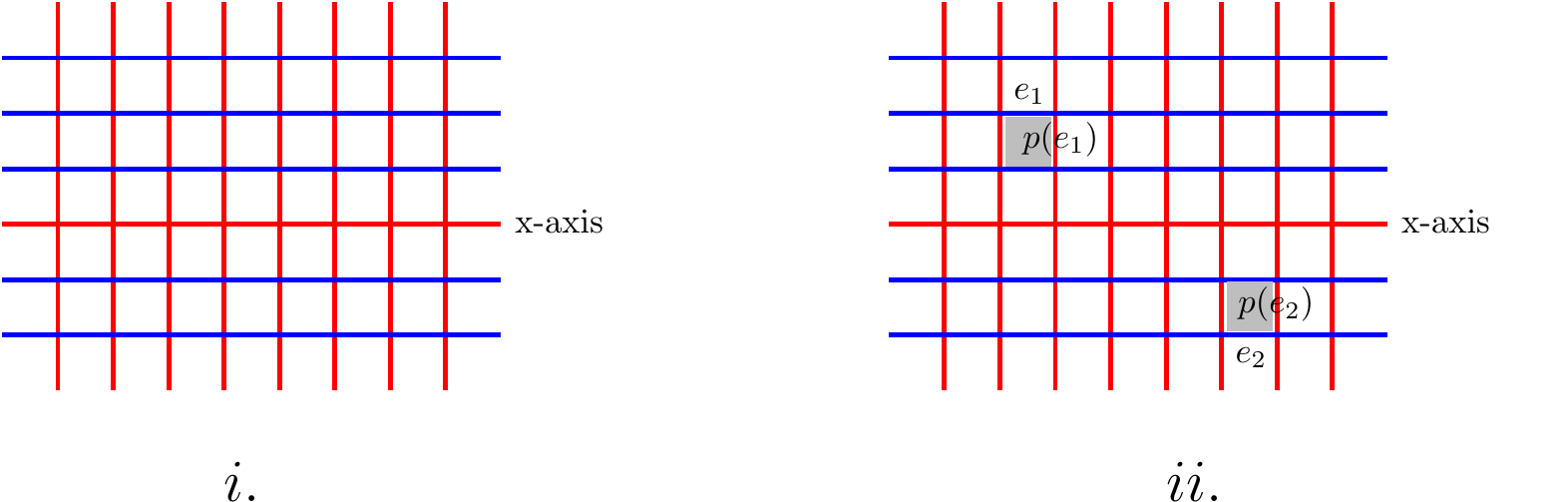}
\caption{$i.$ The red edges are forced to be identity by our choice of the gauge. $ii.$  The plaquettes associated to the edges after gauge fixing.}
\label{f:epe}
\end{figure}

One crucial property of the above action is that it keeps the statistics depending on plaquettes invariant. 
\begin{lem} \label{inv1}The value of  any function $f(Q)$ depending only on the variables $\{\Tr(Q_p)\}_{p\in \cP^{+}_{\Lambda}}$, $f(Q)=f(GQ)$ for any $G\in G(\Lambda)$. More generally, for a loop $\ell=e_1e_2\cdots e_k$ with all edges contained in $\Lambda$ we have $\Tr(Q_{e_1}\cdots Q_{e_k})= \Tr (GQ_{e_1}\cdots GQ_{e_{k}})$.
\end{lem} 

\begin{proof} For any $G\in G(\Lambda), Q\in \cQ,$ and $p=e_1e_2e_3e_4\in \cP^{+}$ we have
$$\Tr((GQ)_p)=\Tr(G(x)Q_{e_1}Q_{e_2}Q_{e_3}Q_{e_4}G(x)^{-1})=\Tr(Q_p)$$
where $x$ is the starting vertex of $e_1$ (and the ending vertex of $e_4$). The result for a general loop $\ell$ follows by an identical argument.
\end{proof}

Now given any $Q\in \cQ,$ we will fix a gauge $G_Q\in G(\Lambda),$ which forces all $(G_QQ)_e=\mathbb{I}$ (the identity matrix) for all $e\in V_n$ as well as on all the edges on the $x$-axis; see Figure \ref{f:epe}. The construction is inductive:\\

$\bullet$ For $x=(0,0)$ define $G_{Q}(x):=S$ where $S$ is any fixed matrix on $O(N)$.\\

$\bullet$ For all $x:=(x_1,0)$ with $x_1>0$, define $G_Q(x):=G_{Q}(0,0)\prod _{i=0}^{x_1-1} Q((i,0), (i+1,0)),$\footnote{The product notation needs a remark since we are in the non-commutative setting. Throughout this article any expression of the form $\prod_{i=k}^\ell A_i$ means $A_kA_{k+1}\ldots A_{\ell}$.}\\

$\bullet$ For all $x:=(x_1,0)$ with $x_1<0$, define $G_Q(x):=G_{Q}(0,0)\prod _{i=0}^{x_1+1} Q((i,0), (i-1,0))$.\\

$\bullet$ For all $x=(x_1,x_2)$ with $x_2>0$, define $G_Q(x):=G_{Q}(x_1,0)\prod_{j=0}^{x_2-1}Q((x_1,j),(x_1,j+1))$.\\

$\bullet$ For all $x=(x_1,x_2)$ with $x_2<0,$ define $G_Q(x):=G_{Q}(x_1,0)\prod_{j=0}^{x_2+1}Q((x_1,j),(x_1,j-1)).$\\

Clearly, this forces $G_QQ$ to be identity on $V_n,$ as well as on all the edges on the $x$-axis; call this set of edges $V'_n$ and the remaining edges $H'_n$. Let $Q(V'_n)$ (resp.\ $Q(H'_n)$) denote the configuration $Q$ restricted to the edges of $V'_n$ (resp.\ $Q$ restricted to the edges of $H'_n$). Let $\cQ_{V'_n}$ be the space of configurations of matrices in $SO(N)$ indexed by edges in $V'_n$ and similarly $\cQ_{H'_n}$, $\cQ_{\cP^+_\Lambda}$ are defined to be the space of configurations of matrices in $SO(N)$ indexed by edges in $H'_n$ and plaquettes in $\cP^+_\Lambda,$ respectively. The next lemma is a consequence of Lemma \ref{inv1}.
%Now suppose $Q$ is distributed according to the Gibbs measure $\mu_{\Lambda, N, \beta}$. Considering the change of variable $(Q(V'_n),Q(H'_n)) \to (Q(V'_n),(G_{Q}Q)(H'_n))$ and 
%Using Lemma \ref{inv1} we have the following:

\begin{lem}
\label{conditional} 
$\{(G_Q Q)_{e}\}_{e\in H_n'}$ has the density proportional to
$$ \exp \biggl( N\beta \sum_{p\in \mathcal{P}_{\Lambda}^{+}} {\rm Tr}(Q_{p})\biggr)$$
with respect to the the measure which is the product of Haar measure on $SO(N)$ on the edges of $H'_n$, and degenerate at identity on the edges in $V_n'.$ 
\end{lem}

\begin{proof}
Consider first the bijective change of variables, $$\Psi_1:\cQ \to (\cQ_{V'_n},\cQ_{H'_n}),$$
given by $\Psi_1(Q)=(\{Q_{e}\}_{e\in V_n'}, \{(G_Q Q)_{e}\}_{e\in H_n'})$. Suppose for the moment that $Q$ is a configuration picked from the product Haar measure on the edges on $\Lambda$, i.e, $\sigma_{\Lambda, N}$ (see \eqref{prodhaar}). Notice that just by a simple conditioning and using the invariance of Haar measure under left and right multiplication, the distribution of $G_QQ$ given $Q(V'_n)$ is again product Haar measure on the edges of $H'_n$. Thus one concludes that the Jacobian for this change of variable is one.

Observe now that from Lemma \ref{inv1} it follows that $\sum_{p\in \mathcal{P}_{\Lambda}^{+}} {\rm Tr}(Q_{p})$ is a function only of the configuration $G_{Q}Q$ and hence only of $\{(G_Q Q)_{e}\}_{e\in H_n'}$ as $G_{Q}Q$ is identity on the other edges. Going back to the scenario where $Q$ is drawn from the measure $\mu_{\Lambda, N,\beta}$ it follows from the above observation that the joint density of $\Psi_1(Q)$ factorises along the two components, immediately yielding the result of the lemma.
\end{proof}

Let $\tilde \cQ$ denote the space of configurations of matrices from $SO(N)$ on $E_{\Lambda}^{+}$ such that all the matrices on the edges of $V'_n$ are identity. Let $\tilde Q=\{\tilde Q_{e}\}_{e\in E_{\Lambda}^{+}} \in \tilde \cQ$ be distributed according to the measure $\tilde{\mu}_{\Lambda, N, \beta}$ whose density is proportional to
$$ \exp \biggl( N\beta \sum_{p\in \mathcal{P}_{\Lambda}^{+}} {\rm Tr}(\tilde Q_{p})\biggr),$$
with respect to the product Haar measure on the edges of $H'_n$. 
Note that $\cQ_{H'_n}$ can be naturally identified with $\tilde \cQ.$ Here $\tilde Q_{p}$ denotes, as before, the product of the matrices along the edges of the plaquette $p$. The following lemma is an easy consequence of the coupling described above and Lemma \ref{inv1}.

\begin{lem}
\label{qq'equiv}
For any loop $\ell=e_1\cdots e_k$ contained in $\Lambda,$ 
$$\E_{\mu_{\Lambda, N, \beta}} \Tr (Q_{e_1}\cdots Q_{e_k})=\E_{\tilde \mu_{\Lambda, N, \beta}} \Tr (\tilde Q_{e_1}\cdots \tilde Q_{e_k}).$$
\end{lem} 

\begin{lem}
\label{l:pfree}
Let $\tilde Q$ be a configuration with law $\tilde \mu_{\Lambda, N, \beta}$. For $p\in \mathcal{P}_{\Lambda}^{+}$, all the variables $ \tilde Q_p$ are independent with density with respect to Haar measure on $SO(N),$ being proportional to $\exp(\Tr(\tilde Q_p)).$
\end{lem}

\begin{proof} By definition $\tilde Q_e=\bI$ for  all $e\in V'_n$. Observe that each of the edges $e$ in $H'_n$ can be uniquely associated with a plaquette $p(e)$ as follows. For an edge $e$ above the $x$-axis, define $p(e)$ to be the the plaquette whose top edge is $e$. For an edge $e$ below the $x$-axis, define $p(e)$ to be the plaquette who bottom edge is $e$; see Figure \ref{f:epe}.
Recall previously since we were only concerned about traces, we chose not to well define the representation of a plaquette i.e. $p=e_1e_2e_3e_4$ was considered to be the same as $e_2e_3e_4e_1$. However in the sequel choosing the representation will be important and we fix the unique representation so that \eqref{inversion1} is true.
 
We now consider the change of variables $\Psi_2: \tilde \cQ \to \cQ_{\cP^+_\Lambda}$  given by, $$\Psi_2(\{\tilde Q_e\}_{e\in H'_n}) = \{\tilde Q_{p(e)}\}_{e\in H'_{n}}.$$

Observe that for $e_{i,j}=((i,j),(i+1,j),$  for $j>0$ we have
\begin{equation}\label{inversion1}
\tilde Q_{e_{i,j}}=\prod_{k=1}^{j}\tilde Q_{p(e_{i,k})}.
\end{equation}
A similar expression holds when $j<0$. Note that by the natural identification, we will also think of $\Psi_2$ as a function on  $\cQ_{H'_n}$. Using invariance of Haar measure on $SO(N)$ on left multiplication notice that if $\{\tilde Q_{p(e)}\}_{e\in H'_n}$ were in fact  distributed according to product Haar measure, then  so is $\{\tilde Q_{e}\}_{e\in H'_n}$. This implies that the Jacobian of this transformation is 1. 
Now as the density  of $\tilde \mu_{\Lambda, N, \beta}$ with respect to product Haar measure depends only on $\{\tilde Q_{p(e)}\}_{e\in H'_n}$ we are done. 
\end{proof}

\begin{remark}\label{decoupling}
From the above, we see that the bijective map $\Psi : \cQ \to (\cQ_{V'_n}, \cQ_{\cP^+_{\Lambda}})$ given by $$\Psi (Q)=(\{Q_{e}\}_{e\in V'_n}, \Psi_2 \circ \Psi_1((\{Q_{e}\}_{e\in H'_n}) ),$$ decouples the coordinates (when $Q$ is sampled from the measure $\mu_{\Lambda, N, \beta}$), where the distribution on $\cQ_{V'_n}$ is product Haar measure and the distribution on $\cQ_{\cP^+_{\Lambda}}$ is an independent product distribution with common marginal described in Lemma \ref{l:pfree}.
\end{remark}

%\red{
%Also we should say $Q'$ or $\tilde Q$ to precisely show that the plaquettes are sampled from a different measure. 
%} 

As we only care about the loop expectations for the rest of this article, whenever we are in the planar setting, using Lemma \ref{qq'equiv} we shall only care about configuration space $\tilde \cQ$ and a configuration $\tilde Q$ drawn from it according to the measure $\tilde \mu_{\Lambda, N, \beta}$.

\section{Asymptotic Freeness of Plaquette Variables}\label{fa}
In this section we use basic tools and techniques from free probability to establish the asymptotic freeness for plaquette variables for a configuration picked from $\tilde \mu_{\Lambda, N, \beta}$. We begin by recalling a notion of convergence for matrices and the notion of a non-commutative probability space.

\subsection{Non-commutative Probability Space and Convergence of Matrices}
We start with the basic definition of convergence for a sequence of random matrices.
\begin{defn}\label{conv1} We say that a sequence ${(A_N)}_{N\in \N}$ of $N \times N$ matrices converge, if the limit $\lim_{N\to \infty} \E[\frac{1}{N}\Tr({A^k_N})]$ exists for all $k\in \N$.
\end{defn}
One of the major motivating questions which initiated the theory of free probability is the following.

\textit{Question:} For any $n\in \N,$ let $(A_n,B_n)$ denote a joint distribution on pairs of matrices of size $n\times n.$ If the marginals $A_n$ and $B_n$ converge, does ${\E[\frac{1}{n}\Tr({Z})]}$ converge for all monomials $Z$ of two symbols and moreover is there a natural way to compute the limits in terms of the marginals?

The following notion of a non-commutative probability space is the natural framework to study this question.

\begin{defn}\label{alg1} A pair $(A, \phi)$ consisting of a unital algebra\footnote{An algebra $A$ is said to be unital if it has  an unit element $1$ such that  $1.x=x.1$ for all $x\in A.$} $A$ and a linear functional $\phi : A \to  \C$ with $\phi (1) = 1$ is called a non-commutative probability space. Often the adjective ``non-commutative" is just dropped. Elements from $A$ are addressed as (non-commutative) random variables; the numbers $\phi(a_1.a_2.\ldots.a_k)$ for such random variables $a_1, \ldots , a_k \in A$ are called moments; the collection  $\{\phi(w)\}_{w\in A(a_1,a_2,\ldots a_k)}$ where $A(a_1,a_2,\ldots a_k)$ is the sub-algebra generated by $a_1,\ldots a_k$ is called the joint distribution of $a_1,\ldots,a_k$.
\end{defn}

All our linear functionals $\phi$ will be tracial i.e.,  $$\phi(ab)=\phi(ba)$$ for any $a,b\in A.$
Let us consider a sequence of non-commutative probability spaces $\sF_{N}=(\cF,\phi_N)$ on the same algebra $\cF$ with different linear functionals $\phi_N.$ \\

\noindent
\textbf{Convergence of algebras:} We say $\{\sF_{N}\}$ converges to $\sF=(\cF, \phi)$ ($\phi$ is a trace functional), if for all $z\in \cF,$
$$\lim_{N\to \infty} \phi_N (z)=\phi(z).$$ 

We now state the abstract definition of freeness even though we will be only concerned with algebras generated by certain matrices and the functional being the normalized trace operator. 
\begin{defn} Given a non-commutative probability space $(\cF,\phi),$ we say sub-algebras $$\cF_1, \cF_2, \ldots \cF_k \subset \cF$$ are jointly free if the following holds: for any monomial of the form $z=a_{i_1}a_{i_2}\ldots a_{i_m},$ where $a_{i_j}\in \cF_{i_j},$ and $\phi(a_{i_j})=0,$ and $i_j \in \{1,2,\ldots, k\}$ for $j=1,2,\ldots, m$ then  
we have $$\phi(z)=0,$$ whenever $i_j \neq i_{j+1}$ for all $j=1,2,\ldots, m-1.$ 
Often we say a collection of elements in the ambient algebra are free if the algebras generated by each of the letters are  jointly free. 
We say the probability sub-spaces $(\cF_1,\phi_n),(\cF_2,\phi_n), \ldots (\cF_k,\phi_n)$ are asymptotically free if the non-commutative probability space $(\{\cF_1,\ldots, \cF_k\},\phi_n)$ jointly converges to $(\{\cF_1,\ldots, \cF_k\},\phi)$ and the limiting subspaces $(\cF_1,\phi),(\cF_2,\phi), \ldots (\cF_k,\phi)$ are free.
\end{defn}

Given any sequence of random matrices $A_N,$ by algebra generated by $A_N$ we would mean the free unital algebra generated by symbols $a,a^{-1}$ modulo the relation $aa^{-1}=a^{-1}a=\bI$ along with the trace functional $\psi_N$ which acts on powers of $a$ in the natural way: for any $k\in \Z,$ $$\psi_N(a^{k})=\frac{1}{N}\E[\Tr(A^k_N)],$$  where the expectation is taken over the law of the random matrix.

\subsection{Algebra of Plaquette Variables}
Throughout the rest of the article we will mostly focus on the following special algebra. 
For any $p\in \cP^+,$ let $q_p,q_p^{-1}$ be symbols and let $\cF$ denote the free algebra generated by symbols $\{q_p, q_p^{-1}\}_{p\in \cP^+}$ modulo the constraints $q_pq_p^{-1}=q_p^{-1} q_p=\bI$, where $\bI$ is the identity of the algebra.

For any $N,$ let $\sF_{N}=(\cF, \langle\frac{1}{N}\Tr(\cdot)\rangle),$ where for any $z=q^{\ep_1}_{p_{1}}q^{\ep_2}_{p_{2}}\ldots q^{\ep_k}_{p_{k}}\in \cF,$ ($\ep_{i}\in \{-1,1\}$), we have the natural definition, 
\begin{equation}\label{finitefree}
\langle\frac{1}{N}\Tr(z)\rangle:= \langle\frac{1}{N}\Tr(\tilde Q^{\ep_1}_{p_{1}} \tilde Q^{\ep_2}_{p_{2}}\ldots \tilde Q^{\ep_k}_{p_{k}})\rangle.
\end{equation}

\begin{ppn}\label{con1} 
 $\{\sF_{N}\}$ converges to $\sF=(\cF, \phi )$  
for a certain trace functional $\phi,$ and moreover the sub-algebras  $\{q_{p},{q_p}^{-1}\}_{p\in \cP^+}$ are free with respect to each other. 
\end{ppn}

 To prove Proposition \ref{con1}  we will use an observation  employed in a related context by Gross-Witten \cite{GW80} and one of the most important results in free probability theory. Recall that  $\tilde Q$ is a configuration of matrices distributed according to the measure $\tilde \mu_{\Lambda,N,\beta}$. Let $\hat \sigma_N$ be the Haar probability measure on the group of orthogonal matrices $O(N).$ Recall that $\sigma_N$ was used to denote the Haar probability measure on the group $SO(N)$.  
The following observation is crucial:   

\begin{lem}[Gross-Witten \cite{GW80}]
\label{varchange}
Let $p_1,p_2,\ldots p_k$  be distinct plaquettes. Sample $S_1, S_2, \ldots S_k$  independently according to the Haar measure $\hat \sigma_{N}$. Then $\{S_{i}{\tilde Q}_{p_i} S^{-1}_{i}\}_{1\le i\le k}$ is  distributed as  $\{{\tilde Q}_{p_i}\} _{1\le i\le k}$. 
\end{lem}
\begin{proof}
The proof follows from the following two facts: 
\begin{itemize}
\item Given any $S_i\in O(N)$, the Haar measure $\sigma_{N}$ is invariant under conjugation by $S_i.$
%\begin{proof} 
This claim is justified as follows. We use the fact that $\hat \sigma_{N}$ is invariant under conjugation by $S \in O(N)$ (by definition of Haar measure) and $SO(N)$ is a normal subgroup of $O(N)$ as well as the fact that $\hat \sigma_{N}|_{SO(N)}$ has total mass $1/2$ and  is the same as $\sigma_N$ multiplied by a factor $1/2.$ 
Thus for any Borel set $\cA\subset SO(N) \subset O(N),$ 
$$\sigma_N (S \cA S^{-1})= 2 \hat \sigma_N (S \cA S^{-1}) =2 \hat \sigma_N (\cA)= \sigma_N (\cA).$$
%\end{proof}
\item  For any two matrices $A,B$ ($B$ is non singular), $$\Tr(A)=\Tr(BAB^{-1}).$$
\end{itemize}
We now show that given any $\{S_i\}_{},$  the quenched distribution of $\{S_i \tilde Q_{p_i} S_i^{-1}\}_{1\le i \le k}$ is the same as $\{\tilde Q_{p_i}\}_{1\le i \le k}$.
Using Lemma \ref{l:pfree} it suffices to prove it for just one plaquette $p_1.$
This follows by considering the change of variables ${\tilde Q}_{p_1}\to S_1{\tilde Q}_{p_1}S_1^{-1}.$ On $SO(N)$ equipped with Haar measure,  this transformation has Jacobian $1$ because of the first fact above.  Since  the Hamiltonian for the Gibbs measure is $N\beta  {\rm Tr}(\tilde Q_{p_1})= N\beta  {\rm Tr}(S_1\tilde Q_{p} S^{-1}_1),$ the proof is complete. 
\end{proof}

\begin{remark}\label{inversion} Using the fact that for any $\tilde Q\in SO(N),$ $\Tr (\tilde Q^{-1})=\Tr( \tilde Q^T)=\Tr(\tilde Q)$ ($\tilde Q^T$ denotes the transpose of $\tilde Q$), and the fact that the Haar measure is invariant under inversion (follows from unimodularity of the group $SO(N)$), it also follows that $ \{{\tilde Q}_{p_i}\} _{1\le i\le k}$ has the same distribution as $\{{\tilde Q}^{-1}_{p_i}\} _{1\le i\le k}.$ 
%\red{check}
\end{remark}

To prove Proposition \ref{con1} we also need a well known result from free probability theory. We refer the interested reader to \cite{speichernote} and the references therein for the necessary background. 

\subsection{Asymptotic Freeness of orthogonally invariant measures.}\label{rev}
Recall from \eqref{finitefree}, that in the setting of matrices the functional is always taken to be expected normalized trace where the underlying measure in our case is the finite $N$ lattice gauge measure.

A random matrix $A_N,$ of size $N \times N,$ is said to be orthogonally invariant, if $$A_N \overset{law}{=}SA_NS^{-1},$$ where $S$ is an independently sampled Haar distributed  matrix from $O(N)$.   
\begin{thm}[\cite{voi91}, Proposition 5.4  \cite{Re15}] \label{free}
Fix any $t\in \N.$ 
Consider for $1\le i\le t,$ $N\times N$ random matrices $A^{(i)}_N$  such that for each $i$: the sequence $\{A^{(i)}_N\}_{N\in \N}$  converge in the sense of Definition \ref{conv1} as $N \to \infty$;  $A^{(i)}_N$ are independent; $A^{(i)}_N$'s are orthogonally invariant ensembles. Then the algebras, generated by $A^{(i)}_N$ are asymptotically free.
\end{thm}
The above statement is a version of  \cite[Proposition 5.4]{Re15} adapted to our setting.
\begin{remark} Note that the hypothesis only assumes convergence of the marginals and not joint convergence. 
\end{remark}

\begin{proof}[Proof of Proposition \ref{con1}]
The proof is now a simple consequence of the already stated results. The convergence of the individual plaquette variables follow from Theorem \ref{t:looplimit}. That these variables  are independent and orthogonally invariant is the content of Lemma \ref{l:pfree} and Lemma \ref{varchange} respectively. Thus an application of Theorem \ref{free} completes the proof.
\end{proof}

\subsection{Cumulants and Non-crossing Partitions}
It is well known that freeness can be characterized in terms of cumulants. We will use this heavily in later sections.
Recall that a non-crossing partition of numbers $\{1,2,\ldots n\}$ is a partition such that it is possible to  add edges between any two points in the same part (viewed as points on the number line embedded in $\R^2$) on the upper half plane which do not cross each other. 

\begin{defn}\label{cum1}Given a non-commutative probability space $(\cF,\phi),$ for any $n$ and $a_1,a_2,\ldots a_n\in \cF$
the generalized cumulants $k_n(a_1,a_2,\ldots,a_n)$ are inductively defined by the relation:
\begin{equation}\label{inver1}
\phi(a_{1}a_2\cdots a_n)=\sum_{\pi\in NC(n)}k_{\pi}(a_{1},a_2,\ldots, a_n),
\end{equation}
where $NC(n)$ denotes the set of all non-crossing partitions of the set $\{1,2,\ldots, n\}$ and  $k_{\pi}(\cdot)$ is the product over the cumulants of the blocks of the partition $\pi$.  
\end{defn}
For e.g. if $n=4$ and $\pi=\{1,2,3,4\}$ then $k_{\pi}(a_1,a_2,a_3,a_4)=k_4(a_1,a_2,a_3,a_4)$ whereas if $\pi=\{\{1,2\},\{3,4\}\}$ then $k_{\pi}(a_1,a_2,a_3,a_4)=k_2(a_1,a_2)k_2(a_3,a_4).$ 

The next theorem gives an equivalent characterization of freeness in terms of cumulants.

\begin{thm}
\label{alternate}\cite{speichernote} 
Consider a non-commutative probability space $(\cF,\phi)$. The following conditions are equivalent:
\begin{enumerate}
\item $a_1,\ldots,a_\ell \in \cF$ are free,
\item Mixed cumulants vanish, i.e., $k_n(a_{i(1)},\ldots,a_{i(n)}) = 0,$ $n\in \N$ for any $i(1),i(2),\ldots,i(n)\in \{1,2,\ldots,\ell\}$ whenever there exists $1 \le p,q \le n$ with $i(p)\neq i(q)$.
\end{enumerate}
\end{thm}
For precise definitions and an excellent exposition of combinatorics related to non-crossing partition, see \cite[Section 5]{speichernote}.

We also say $\pi_1  \preceq \pi_2,$ for $\pi_1,\pi_2 \in NC(n),$ if every part of $\pi_1$ is a subset of a part of $\pi_2.$ Clearly this defines a partial order with the smallest element and largest elements being $(\{1\},\{2\},\ldots, \{n\})$ and $(\{1,2,\ldots ,n\})$ which we denote by $\mathbf{0}_n$ and $\mathbf{1}_n$ respectively. Often when $n$ is clear from the context, we will drop the subscripts. 

Several multiplicative functions are  important in the context of free probability. We will be only using the M\"obius function denoted by  $\mu(\pi_1,\pi_2)$ for $\pi_1 \preceq \pi_2$ defined multiplicatively using the definition $\mu (\mathbf{0}_n,\mathbf{1}_n)=(-1)^{n-1}C_{n-1},$ ($C_n$ is the $n^{th}$ Catalan number; $C_0=1,C_1=1, C_2=1$ etc.). Moreover, 
\begin{equation}\label{convolution}
\sum_{\pi_1 \preceq \tau \preceq \pi_2}\mu(\pi_1,\tau)=\mathbf{1}(\pi_1=\pi_2).
\end{equation}
For precise definition of $\mu(\pi_1, \pi_2)$ with $\pi_1 \preceq \pi_2,$ and a proof of the above equality  see \cite[Section 5]{speichernote}. 
We have the following M\"obius inversion formula:
\begin{equation}\label{mobius}
k_{n}(a_1,a_2,\cdots,a_n)=\sum_{\pi\in NC(n)}\phi_{\pi}(a_{1}a_2\cdots a_n) \mu(\pi,\mathbf{1}_n),
\end{equation}
where  $\phi_{\pi}(\cdot)$ is the product of $\phi$ over the blocks of the partition $\pi$ similar to how $k_{\pi}(\cdot)$ was defined earlier.

To use Theorem \ref{alternate} we need some notations regarding non-crossing partitions.
Recall the non-commutative probability space $\sF=(\cF,\phi)$ from Proposition \ref{con1}.
Given a monomial, $$z=q^{\ep_1 k_1}_{p_1}q^{\ep_2 k_2}_{p_2}\ldots q^{\ep_\ell k_\ell}_{p_\ell}$$ where $\ep_{i}=\pm1$ for all $i=1,2,\ldots, \ell$ and $p_i\neq p_{i+1}$; let $NC(z)$ denote the set of non crossing partitions of $\{1,2,\ldots \ell\}$ with the property that any two indices $i,j$ belong to the same part of $\pi$ only if $p_i=p_j$. \ Also  let 
\begin{equation}\label{monomial100}
\tilde Q(z)= \tilde Q^{\ep_1 k_1}_{p_1}\tilde Q^{\ep_2 k_2}_{p_2}\ldots \tilde Q^{\ep_\ell k_\ell}_{p_\ell}.
\end{equation}
%textcolor{red}{What does this mean???}

The following proposition evaluates the limiting expected trace of $\tilde Q(z).$ 

\begin{ppn}
\label{con2} 
For any monomial $z$,
$$ \lim_{N\to \infty} \frac{1}{N} \E \Tr(\tilde Q(z))= \sum_{\pi \in NC(z)}c(\pi,z) \beta ^{d(\pi,z)},$$
where $c(\pi,z), d(\pi,z)$ are constants depending only on $\pi$ and $z$.
\end{ppn}

\begin{proof} 
The proof follows from Proposition \ref{con1}, \eqref{inver1}, Theorem \ref{alternate} and the fact that all cumulants are polynomials in $\beta$, (this is proved below in Lemma \ref{poly100}). 
\end{proof}

\begin{lem}
\label{poly100}
Let $a=q_p$ denote a plaquette variable in the algebra $(\mathcal{F}, \phi)$ from Proposition \ref{con1}. Then we have the following:
\begin{enumerate}
\item[(i)] For any $m>0,$ and $s_i\in \Z,$ $k_m(a^{s_1},a^{s_2},\ldots, a^{s_m})$  is a polynomial in $\beta$ divisible by $\beta^{\alpha}$ where $\alpha=|\sum_{i=1}^ms_i|.$
\item[(ii)] Moreover, if $s_i\in \{ \pm1 \},$ then for every $\pi$ the leading term of $k_{\pi}(a^{s_1},a^{s_2},\ldots, a^{s_m})$ is $\beta^{m}$ with coefficient $\mu(\mathbf{0},\pi).$
\item[(iii)] Under the hypothesis of part (ii),  there exists $\pi\in NC(m)$ such that each part is a singleton or a pair $(a,a^{-1}),$
 with $k_{\pi}(a^{s_1},a^{s_2},\ldots, a^{s_m}) =\beta ^{\alpha}(1+\beta P(\beta)),$ for some polynomial $P(\beta)$.
\end{enumerate}
\end{lem}

\begin{proof} Using \eqref{mobius} we get,
\begin{equation*}k_m(a^{s_1},a^{s_2},\ldots, a^{s_m})=\sum_{\pi\in NC(m)}\phi_{\pi}(a^{s_1},a^{s_2},\ldots, a^{s_m}) \mu(\pi,\mathbf{1}).
\end{equation*}
Recall from Theorem \ref{pla100} that $\phi(a^{s_1}a^{s_2}\cdots a^{s_m})$ is $1$ if $\alpha=0,$ $\beta$ if $\alpha=1$ and $0$ otherwise. Now for any $\pi\in NC(m),$ let $P_1,P_2,\ldots, P_k$ be the parts. Let $\alpha_i= |\sum_{ j \in P_i} s_j|.$   For any $\pi$ if any of the $\alpha_i>1$ then $\phi_{\pi}(a^{s_1},a^{s_2},\ldots, a^{s_m})$ vanishes.  Otherwise the contribution is $\beta ^{\sum_{i}\alpha_i}.$  Since $\sum_{i}\alpha_i\ge \alpha,$ part (i) follows. 

For part (ii), note that  for any $n$ the coefficient of $\beta^n$ in $k_n(a^{s_1},a^{s_2},\ldots, a^{s_n})$  is $\mu(\mathbf{0}_n,\mathbf{1}_{n}).$ using \eqref{mobius} as the only contribution to the term $\beta^n$ comes from $\pi=\mathbf{0}_n$. 
Thus the coefficient of $\beta^m$ in $k_{\pi}(a^{s_1},a^{s_2},\ldots, a^{s_m})$ is $\mu(\mathbf{0},\pi),$ using the multiplicative nature of of M\"obius function, \cite[Section 4.2]{speichernote}.

For part (iii), note that given any word of $a$ and $a^{-1}$ one can successively delete consecutive pairs $(a, a^{-1})$
or $(a^{-1}, a)$ to end up with $\alpha$ consecutive $a$'s or $a^{-1}$'s depending on whether $\sum_i s_i$ is positive or not. Form $\pi$ by deleted pairs and the remaining elements forming singleton sets. 
Clearly by construction $\pi\in NC(m).$ Also $$k_\pi(a^{s_1},a^{s_2},\ldots, a^{s_m}) = k_2^{(m-\alpha)/2}k^{\alpha}_1,$$ where $k_2=k_2(a,a^{-1})=k_2(a^{-1},a)=1-\beta^2$ and $k_1=k_1(a)=k_1(a^{-1})=\beta$ (the equalities are easy to check from the definition of cumulants and Theorem \ref{pla100}). This completes the proof.
 \end{proof}
 
\section{Loop Expectations are Polynomials in $\beta$}
\label{s:poly}
Using the machinery developed in the previous section, we complete the proof of Theorem \ref{t:poly}. Recall the definition of the configuration space $\tilde \cQ$ and the measure $\tilde \mu_{\Lambda, N, \beta}$ from Section \ref{gf}. Throughout this section we shall work with a configuration $\tilde Q$ sampled from this measure space. For this section $\langle \cdot \rangle$ will denote the expectation with respect to $\tilde \mu_{\Lambda, N, \beta}$. 

We begin by proving the following proposition which states that any loop can be written as product of plaquettes. Consider the coupling induced by the map $\Psi$ (see Remark \ref{decoupling}) between the configurations $Q$ and $\tilde{Q}$.
%\red{notational issue} between the matrix ensembles $\cQ$ and $\tilde \cQ_{\cP^{+}_{\Lambda}}.$
\begin{ppn}\label{loop} 
For any loop $\ell,$ there exists plaquettes $p_{1},\ldots, p_{k}$ not necessarily distinct, such that under the above coupling,
$$W_{\ell}= \Tr \prod_{i=1}^{k} \tilde Q_{p_{i}}.$$ 
\end{ppn} 

Recall that under our convention for non-commutative products: $\prod_{i=1}^{k} \tilde Q_{p_{i}}= \tilde Q_{p_{1}} \cdots \tilde Q_{p_{k}}$; this notation will be used throughout the proof as well.

\begin{proof}
Take $\Lambda$ sufficiently large such that all the edges of the loop $\ell$ is contained in $\Lambda$. Let $\ell=e_1e_2\cdots e_{k}$. Now for any edge $e_i$ either $e_i\in V'_n$ in which case $\tilde Q_e$ is identity or otherwise 
by \eqref{inversion1} for any edge $e_{i}\in H'_n$ we can write $\tilde Q_{e_{i}},$ as a product of plaquette variables.
% \red{optimize notation to minimize repetition}. 
Thus we have $\tilde Q_{e_{i}}= \prod_{j=1}^{k} \tilde Q_{p_{i_j}},$
for some plaquette variables $\tilde Q_{p_{i_1}},\tilde Q_{p_{i_2}},\ldots, \tilde Q_{p_{i_r}}$ by \eqref{inversion1}).
It follows that $\tilde{Q}_{\ell}=\prod_{i=1}^{k}\prod _{j=1}^{i_{r}} \tilde Q_{p_{i_r}}$. As trace is invariant under Gauge fixing it follows that $W_{\ell}= \Tr \tilde{Q}_{\ell}$ and thus the result follows.
%
% $\tilde W_\ell=\Tr \prod_{j=1}^{k} \tilde Q_{p_{i_j}}.$
% Since $\tilde W_{\ell}=W_{\ell}$ we are done. 
\end{proof}

Theorem \ref{t:poly} is now almost immediate.

\begin{proof}[Proof of Theorem \ref{t:poly}]
Fix any loop $\ell$. From Proposition \ref{loop} it follows that there exists a monomial (of finite degree) $Z=Z(q_{p_1}, q_{p_2}, \ldots , q_{p_{k}})$ for some plaquettes $p_1, p_2,\ldots ,p_{k}$ such that 
$$w(\ell,\beta)= \lim_{N\to \infty} \frac{ \E \Tr (\tilde Q(Z))}{N}.$$
The theorem now follows from Lemma \ref{con2}.
\end{proof}

\section{Disjoint Loops are Asymptotically Free}
\label{s:loopfree}
We shall prove Theorem \ref{t:simple} in this section. Observe that one natural way of trying to prove that limiting loop expectation of a loop with area $k$ is $\beta^{k}$ is the following. Negatively deform the simple loop $\ell$ at a corner to obtain a simple loop of a smaller area, use asymptotic freeness and induction. Our proof follows the same general idea, however, we establish something stronger along the way, namely that two disjoint loops are asymptotically free. More precisely, we prove the following result.

%In this section we prove the following lemma:
\begin{ppn}\label{free120}
For two simple loops  $\ell_1$ and $\ell_2$ which do not share an edge,  and have disjoint interiors thought of as curves in $\R^2$,  the matrices $Q_{\ell_1}$ and $Q_{\ell_2}$ are asymptotically free.
\end{ppn}

For the proof of this result we shall employ a non-standard gauge fixing, different from the axial gauge fixing introduced in Section \ref{gf}. We move now towards the formal definitions. We shall work with two fixed loops $\ell_1$ and $\ell_2$ satisfying the hypothesis of Proposition \ref{free120}.

\subsection{A Different Gauge Fixing}
Recall that in the axial Gauge fixing in Section \ref{gf}, we forced the matrices on edges of a comb graph to be identities. Here we shall introduce a new Gauge fixing which will force the matrices on a certain path (depending on the loops $\ell_1$ and $\ell_2$) to be identity. We start with the following standard topological fact whose proof we skip:

Given $\ell_1$ and $\ell_2$ satisfying the hypothesis of Proposition \ref{free120} there exists a simple bi-infinite path $\cP$
such that $\Z^2\setminus \cP$ contains two connected components $C_1, C_2$ and for $i=1,2,$ $$\ell_i \subset  C_i \cup \cP.$$
Moreover for any large enough box $\Lambda_{n}=[-n,n]^2,$
$\cP_n=\cP\cap \Lambda_n$ is itself a connected simple path and $\Lambda_{n}\setminus \cP_n$ contains two connected components $C_{1,n}$ and $C_{2,n}$ such that for  $i=1,2,$ 
$$\ell_i \subset  C_{i,n} \cup \cP_n.$$

Let us now fix a path $\cP$ and a sufficiently large box $\Lambda_n$ satisfying the above properties. Suppose $|\cP_n|=m$ where $|\cP_n|$ denote the number of vertices in $\cP_n$.There is a natural graph isomorphism between  $\cP_n$ and any interval in $\mathbb{Z}$ of length $m$; we shall choose a suitable interval according to our notational convenience.

Recall from Section \ref{gf} that for any gauge function $G$
and any loop $\ell$ starting and ending at a vertex $x,$
$$GQ_{\ell}=G(x)Q_{\ell}G^{-1}(x).$$

We now fix a gauge function $G=G_{\cP_n}$ which will force all the edges of $\cP_n$ to be identity. Define the Gauge function $G$ as follows. Let $\theta$ be the natural isomorphism from the line graph on $[-a,b]\cap \Z$ to $\cP_n$ where $a,b\in \N$ such that $a+b=m-1$. Set $G(\theta (0))=\mathbb{I}$ and for any integer $k\in [1,b]$ define $G(\theta (k))=G(0)\prod_{i=1}^{k}Q(\theta (i-1),\theta (i)).$ Define $G(\theta(k))$ similarly for integers $k\in [-a,-1]$. For any $x \notin \cP_n$ we define $G(x)=\mathbb{I}$. Denote by $E' =E(\Lambda_n)\setminus E(\cP_n)$ the set of all edges in $\Lambda_n$ that are not in $\cP_n$. Naturally partition $E'$ into two parts, $E_{1,n}$ and $E_{2,n}$ where $E_{i,n}$ is the set of edges in $E'$ that are incident on the component $C_{i,n}$. Let $\cQ_{E_{i,n}}$ (resp.\ $\cQ_{\cP_n}$) denote the space of $SO(N)$ matrix configurations indexed by the edges of $E_{i,n}$ (resp.\ the edges of $\cP_n$). For a configuration $Q\in \cQ$, let $Q_{\cP_n}$ denote its restriction to the edges of $\cP_n$. For $i=1,2$, let $\hat{Q}_{i}$, denote the configuration $GQ$ restricted to the edges of $E_{i,n}$. Putting the above together, let $\hat Q$ denote the configuration which is $\hat{Q}_{i}$ on $E_{i,n}$ and identity on the edges of $\cP_n$. As before, for a plaquette $p\in \cP_{\Lambda}^{+}$, let $\hat{Q}_{p}$ denote the product of $\hat{Q}$-matrices along edges of $p$. 
We have the following lemma.

\begin{lem}
\label{l:hatdist}
Let $Q$ be distributed according to the measure $\mu_{\Lambda, N, \beta}$ and $\hat{Q}$ be as above. Then $\hat{Q}$ is independent of $Q_{\cP_n}$ and has density proportional to
$$ \exp \biggl( N\beta \sum_{p\in \mathcal{P}_{\Lambda}^{+}} {\rm Tr}(\hat Q_{p})\biggr)$$
with respect to the product of Haar  measure on $SO(N)$ along edges in $E_{1,n}\cup E_{2,n}$. 
\end{lem}

\begin{proof}
This proof is similar to the proof of Lemma \ref{conditional}. Consider the bijective map $\Phi: \cQ\to (\cQ_{\cP_n}, \cQ_{E_{1,n}}, \cQ_{E_{2,n}})$ defined by
\begin{equation}\label{decoupling1}
\Phi (Q)=(Q_{\cP_n},\hat Q_1,\hat Q_2).
\end{equation}
The fact that $\Phi$ is bijective is easy to check since the inverse map is clear from the definition of the gauge function.
As in the proof of Lemma \ref{conditional}, we observe that if $Q$ is distributed according to product Haar measure then so is $(Q_{\cP_n},\hat Q_1,\hat Q_2)$. To see this notice that  $$\{(GQ)_{(x,y)}\}_{(x,y)\notin E(\cP_n)}=\{G(x)Q(x,y)G(y)^{-1}\}_{(x,y)\notin E(\cP_n)}.$$ When $Q$ is distributed under product Haar measure, $Q(\cP_n)$ is independent of $\{Q(x,y)\}_{(x,y)\notin E(\cP_n)}.$ The proof of the above claim now follows from the fact that $G(\cdot)$ is a deterministic function of $Q(\cP_n)$ and that the Haar measure on $SO(N)$ is invariant under conjugation by any matrix $S \in O(N)$. As before this implies that the Jacobian of the transformation $\Phi$ is $1$. Now arguing as in the proof of Lemma \ref{conditional}, we conclude that $Q_{\cP_n}$ is distributed according to product Haar measure on the edges of $\cP_n$; $\hat Q$ is independent of $Q_{\cP_n}$ and the distribution of $\hat{Q}$  with respect to product Haar measure on the edges of $E'$ has density $$ \exp \biggl( N\beta \sum_{p\in \mathcal{P}_{\Lambda}^{+}} {\rm Tr}(\hat Q_{p})\biggr).$$
\end{proof}

Observe that the following domain Markov property is almost immediate from Lemma \ref{l:hatdist}.

\begin{lem}
\label{l:dmprop}
In the above setting $\hat{Q}_1$ and $\hat{Q}_2$ are independent. 
\end{lem}

\begin{proof}
This is immediate from observing that each for each plaquette $p$, the edges in $p$ either is disjoint from $E_{1,n}$ or is disjoint from $E_{2,n}$ and hence the density in Lemma \ref{l:hatdist} factorises. 
\end{proof}

For a loop $\ell$, as before let us denote by $\hat{Q}_{\ell}$ the product of matrices from the configuration $\hat{Q}$ along the edges of $\ell$. To prove the asymptotic freeness in Proposition \ref{free120} we shall again invoke Theorem \ref{free}. To this end, we need the following lemma.

%\begin{lem}
%\label{free200}
%The matrices  
%$\hat Q_{\ell_1}$ and $\hat Q_{\ell_2}$ are asymptotically free.
%\end{lem}

\begin{lem} 
\label{l:inv2}
For $i=1,2$ and any deterministic matrix $S\in O(N)$ the law of $S\hat Q_{\ell_i} S^{-1}$ is the same as that of $\hat Q_{i}.$
\end{lem}
\begin{proof} 
We will only discuss the case $i=1$ since the details for the other case are similar.  
The proof for $i=1$ will be split into two cases:
\begin{enumerate} 
\item
$x_1,$ the starting point of $\ell_1$ in on $\cP_n$. 
 \item $x_1,$ the starting point of $\ell_1$ in not on $\cP_n$. 
 \end{enumerate}
In both cases we slightly modify the gauge function $G:=G_{\cP_n}$ for our purpose. 
 
 Case (1). Choose $a,b$ such that the isomorphism $\theta$ between $[-a,b]\cap \Z$ and $\cP_n$ satisfies $\theta(0)=x_1$. Choose $G'(\theta(0))=S$, and define the rest of the Gauge function $G'$ exactly the same as the definition of $G$ above. Recall we chose $G(0)=\mathbb{I}$ before. We denote the bijective map analogous to the one in the proof of Lemma \ref{l:hatdist} by $\Phi'$. Thus we get $$\Phi' (Q)=(Q_{\cP_n}, Q'_1, Q'_2)$$ 
where for $i=1,2$, we take $ Q'_i$ to be the restriction of the configuration $G'Q$ to the edges of $E_{i,n}$. Following the same arguments as in the proof of Lemma \ref{l:hatdist} and Lemma \ref{l:dmprop} we conclude that  $Q'_1, Q'_2$  are independent of each other and $Q_{\cP_n}$ and have the same law as $\hat Q_1, \hat Q_2$. Now notice that $Q'_{\ell_1}=SQ_{\ell_1}S^{-1}$ while $\hat Q_{\ell_1}=Q_{\ell_1}$ (this is where we need $x_1$ to map to $0$.)
Hence $Q'_{\ell_1}=S \hat Q_{\ell} S^{-1}.$ Thus we are done from the aforesaid equality of law of $\hat Q_1$ and $Q'_{1}.$

Case (2). In this case we define $G'$ to be equal to $G$ everywhere except at $x_1$ where we define $G'(x_1)=S.$ Recall that $G(x_1)=\mathbb{I}.$ The steps of the proof in this case are now verbatim as in Case (1). 
\end{proof}

Combining Lemma \ref{l:inv2} and Lemma \ref{l:dmprop} gives us the following.

\begin{lem}
\label{free200}
The matrices  
$\hat Q_{\ell_1}$ and $\hat Q_{\ell_2}$ are asymptotically free.
\end{lem}

\begin{proof}
Note that both $\hat Q_{\ell_1}$ and $\hat Q_{\ell_2}$ converge in the sense of Definition  \ref{conv1} as the same is true for $Q_{\ell_1}$ and $Q_{\ell_2}$ by Theorem \ref{t:looplimit} and the fact that trace is invariant under conjugation. Now using Lemma \ref{l:dmprop}, Lemma \ref{l:inv2} and Theorem \ref{free} we conclude that they are asymptotically free.
\end{proof}

%\begin{proof}
%
%We notice the following fact:
%
%\begin{itemize}
%\item (Domain Markov Property) $\hat Q |_{C_1}$ and $\hat Q |_{C_2}$ are independent just because the density factorizes. To see this note that by definition, any plaquette $p$ is either entirely contained in the subgraph on $\lambda_n$  induced on $C_{1,n} \cup \cP_n$ or $C_{2,n} \cup \cP_n.$ Thus 
%\end{itemize}
%We now prove the following lemma:

%We now show that the law of the random matrix $\hat Q_{\ell_1}$
%Moreover by choosing a gauge $G'$ which agrees with $G$  every where except $x_1$ where we choose $G(x_1)=S$ for any $S\in O(n)$ we see that 
%\begin{align*}
%G'Q_{\ell_1}&=SGQ_{\ell_1}S^{-1}\\
%&=S\hat Q_{\ell_1}S^{-1}\\
%\end{align*}
%Using invariance of Haar measure on $SO(n)$ under conjugation by elements of $O(n)$ (see proof of Lemma \ref{varchange} \red{may be we can make this a small lemma}) it follows that  $GQ$ and $G'Q$ have the same density with respect to product Haar measure by  Remark \ref{conditional}. 
%Thus it follows that $G'Q_{\ell}$ and $\hat Q_{\ell}$ have the same law. Hence in particular the law of $\hat Q_{\ell}$ is invariant under conjugation by elements of $O(N).$
%To finish the proof of Lemma \ref{free200},

%\end{proof}
Finally we are ready to prove Proposition \ref{free120}.

\begin{proof}[Proof of Proposition \ref{free120}]
Note that under the coupling of the configuration $Q$ and $\hat Q$ through the map $\Phi,$
\begin{align*}
(\hat Q_{\ell_1},\hat Q_{\ell_2})&=(G(x_1)Q_{\ell_1}G^{-1}(x_1),G(x_2) Q_{\ell_2}G^{-1}(x_2)), \mbox{ and hence,}\\
(Q_{\ell_1},Q_{\ell_2})&=(G^{-1}(x_1)\hat Q_{\ell_1}G(x_1),G^{-1}(x_2)\hat Q_{\ell_2}G(x_2)),
\end{align*}
 where $x_1$ and $x_2$ are the starting points of $\ell_1$ and $\ell_2$ respectively.  
Since $G(\cdot)$ is a deterministic function of  $Q_{\cP_n}$, by Lemma \ref{l:hatdist}, the aforesaid orthogonal invariance and independence of $\hat Q_{\ell_1}$ and $\hat Q_{\ell_2}$, we see that the same is true for $ Q_{\ell_1}$ and $ Q_{\ell_2}$ as well.
The remaining part of the proof is the same as that for Lemma \ref{free200}.
\end{proof}

%\subsection{Proof of Theorem \ref{t:simple}}
We are now in a position to compute the limiting loop expectations for any simple loop and conclude the proof of Theorem \ref{t:simple}.

%
%\section{Proofs of Theorem \ref{t:simple} and Theorem \ref{t:poly}}
%\label{s:proof}
%In this section we deduce Theorem \ref{t:simple} from Theorem \ref{t:plaquette} and also prove Theorem \ref{t:poly}. Recall the definition of the configuration space $\tilde \cQ$ and the measure $\tilde \mu_{\Lambda, N, \beta}$ from Section \ref{gf}. Throughout this section we shall work with a configuration $\tilde Q$ sampled from this measure space. For this section $\langle \cdot \rangle$ will denote the expectation with respect to $\tilde \mu_{\Lambda, N, \beta}$. 
%Also let  $\tilde W_{\ell}=\Tr \tilde Q_{\ell}.$ \red{not needed as trace is invariant under gauge fixing}
%
%
%
%We finish this section by proving Theorem \ref{t:simple}. 

\begin{proof}[Proof of Theorem \ref{t:simple}]
We prove this by induction on the area of the simple loop $\ell$. For simple loops of area one, i.e.  plaquettes, this is the content of Theorem \ref{t:plaquette}.
Fix a simple loop $\ell$ with area larger than one, and assume that we have established the result for all simple loops of smaller area. Let $e_1=((x_1+1,x_2),(x_1,x_2))$ be the edge in $\ell$ such that $\ell$ does not intersect the vertical line $X=x_1-1$ and the part of line $X=x_1$  below the line $Y=x_2$ i.e., it is the bottom most among the leftmost edges in $\ell.$ 
By Remark \ref{inversion} we assume without loss of generality that $\ell$ is in clockwise orientation.
Let us now start the loop from the point $(x_1+1,x_2)$. See Figure \ref{f:epe2} and consider the plaquette $p_0=e_1e_2e_3e_4$ as depicted in the figure.   
Thus $\ell=e_1e_2\ell_1$ where $\ell_1$ is a simple path from $(x_1,x_2+1)$ to $(x_1+1,x_2). $ Now we consider the following two cases:

\begin{enumerate}
\item $\ell_1$ does not intersect $(x_1+1,x_2+1).$
\item $\ell_1$ intersects $(x_1+1,x_2+1).$ In this case let $\ell_1=\ell_2\ell_3$ where the end point of $\ell_2$ and the beginning point of is $(x_1+1,x_2+1).$ Moreover notice that since $\ell_1$ is simple, so are $\ell_2$ and $\ell_3.$
\end{enumerate}

Note that in the first case 
\begin{align}\label{case1}
Q_{\ell}&=Q_{p_0}Q^{-1}_{e_4}Q^{-1}_{e_3}Q_{\ell_1},\\
\nonumber
&=Q_{p_0}Q_{\ell'},
\end{align}
where $\ell'=e_4^{-1}e_3^{-1}\ell_1$ is a simple loop.

In the second case 
\begin{align}\label{case2}
Q_{\ell}&=Q_{p_0}Q^{-1}_{e_4}\left(Q^{-1}_{e_3}Q_{\ell_2}\right)\left(Q_{\ell_3}Q^{-1}_{e_4}\right) Q_{e_4},\\
\nonumber
&=Q_{p_0}Q^{-1}_{e_4}Q_{\hat \ell}Q_{\tilde \ell}Q_{e_4},
\end{align}
where $\hat \ell=e_3^{-1}\ell_2$ and $\tilde \ell=\ell_3e_4^{-1}$ are respective simple loops. We shall apply the axial Gauge fixing introduced in Section \ref{gf}. By a simple translation we assume  $(x_1+1,x_2+1)=0.$ Recall the configuration space $\tilde{\cQ}$ and the measure $\tilde{\mu}$ from Section \ref{gf}. Using the same notation as before observe that 
$$\E_{\mu}(\Tr Q_{\ell})= \E_{\tilde \mu}(\Tr \tilde Q_{\ell})= \E_{\tilde \mu}(\Tr \tilde Q_{p_0}\tilde Q_{\ell'})$$
in case 1 and 
$$\E_{\mu}(\Tr Q_{\ell})= \E_{\tilde \mu}(\Tr \tilde Q_{\ell})= \E_{\tilde \mu}(\Tr \tilde Q_{p_0}\tilde Q_{\hat \ell} \tilde Q _{\tilde \ell})$$in case 2. Observe that the component loops are of all simple smaller area. It is also easy to observe that ${\rm area}(\ell)= 1+{\rm area}(\ell')$ (in case 1) and ${\rm area}(\ell)= 1+{\rm area}(\hat \ell)+ {\rm area}(\tilde \ell)$ (in case 2). An application of asymptotic freeness together with induction will complete the proof. 

\begin{figure}[ht]
\centering
\includegraphics[width=.25\textwidth]{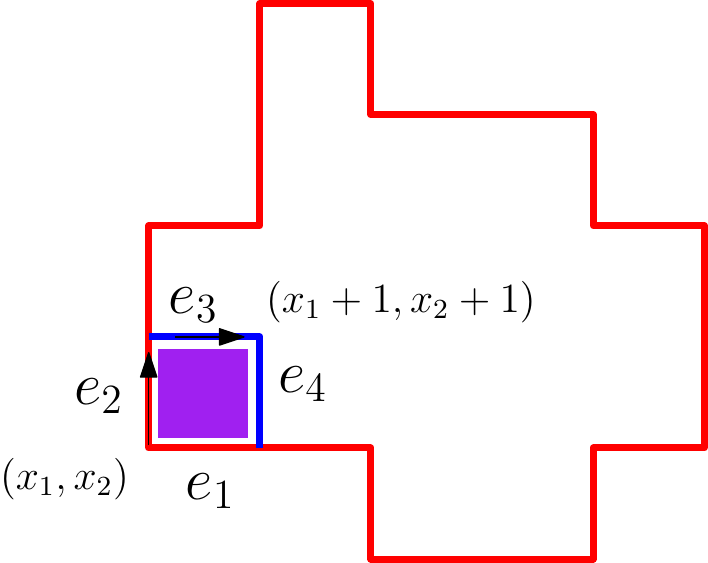}
\caption{The inductive step in the proof of Theorem \ref{t:simple}.}
\label{f:epe2}
\end{figure}
%Formally using,
%$$\E_{\mu}(\Tr Q_{\ell})=\E_{\mu}(\Tr GQ_{\ell})= \E_{\tilde \mu}(\Tr \tilde Q_{\ell})= \E_{\tilde \mu}(\Tr \tilde Q_{p_0}\tilde Q_{\ell'}).$$
%We now argue separately for the two cases discussed above:
Formally we need to argue separately in the two cases. 
\begin{itemize}
\item Case (1). This case is more straightforward, and just the asymptotic freeness of plaquette variables from Section \ref{fa} suffices. Recall from \eqref{inversion1} we can write $\tilde Q_{\ell'}$ as product of plaquettes disjoint from $p_0$:
%In the first case  $\E_{\mu}(\Tr Q_{\ell})=\E_{\tilde \mu}(\Tr \tilde Q_{p_0}\tilde Q_{\ell'})$. Now note that using \eqref{inversion1} 
$$\tilde Q_{\ell'}=\prod_{i=1}^k \tilde Q^{\epsilon_i}_{p_i}$$ for some $k,$ $p_1,\ldots, p_k$ and $\epsilon_1,\ldots, \epsilon_k \in \{1,-1\},$ where $p_i\neq p_0$ for all $i=1,\ldots, k.$

Recall the trace functional $\phi$ and algebra of plaquette variables $\cF$ from Proposition \ref{con1}. Define $q_{\ell}$ and $q_{\ell'}$ to be the natural element in the algebra $\cF$ corresponding to the matrices $\tilde Q_{\ell}$ and $\tilde{Q}_{\ell'}$ respectively. Using Proposition \ref{con1} it now follows
\begin{align*}\phi (q_\ell)= \phi(q_{p_0})\phi(q_{\ell'}) 
%\red{\text{is it defined?}}
\end{align*}
\item Case (2).
Arguing as in case (1) we obtain 
\begin{align*}\phi (q_\ell)= \phi(q_p)\phi(q_{\hat\ell}q_{\tilde\ell})
\end{align*}
Now note that $\hat \ell \tilde \ell$ is itself a loop which is not  quite simple but is a concatenation of two simple loops $\hat \ell$ and $\tilde \ell$ both simple and starting from $(0,0)$ and disjoint otherwise and satisfies the hypothesis of  Proposition \ref{free120}. Using Proposition \ref{free120} we get $$\phi(q_{\hat \ell}q_{\tilde \ell})=\phi(q_{\hat \ell})\phi(q_{\tilde \ell}).$$
\end{itemize}
At this point we notice that in the first case $\rm{area}(\ell')=\rm{area}(\ell)-1$ whereas in the second case $\rm{area}(\hat \ell)+\rm{area}(\tilde \ell)=\rm{area}(\ell)-1.$ Thus in both cases we can finish the proof by induction on the area of the loop. 
%
% $\ell$ admits a decomposition as $\ell =p\ell_1$ where $\ell_1$ is simple and $\ell_1 \cap p \subset \{e_1,e_2\}
%$ (thinking of the loops as collection of edges)
%where $e_1=((x_1,x_2+1),(x_1+1,x_2+1)), e_2=((x_1+1,x_2+1),(x_1+1,x_2)),$ (notice that the gauge fixing forces $ Q_{e_1}=Q_{e_2}= \bI.$ 
%Thus $Q_{\ell}=Q_{p}Q_{\ell_1}$ (recall the notation from Section \ref{wil}).
%Now by the proof of Proposition \ref{loop}, and construction of $\ell_1$ and the gauge,
%$$Q_{\ell_1}=  \prod_{j=1}^{m} Q_{p_{i_j}},$$ for some $m,$ 
%where all the $Q_{p_{i_j}}'s$ are distinct from $Q_p.$
% 
% Thus  by Lemma \ref{con1} the variables  $Q_{{p}},Q_{\ell_1}$ are asymptotically free and hence,
%
%$$\lim_{N\to \infty} \frac{\langle  W_{\ell} \rangle}{N}= \left[\lim_{N\to \infty} \frac{\langle   Q_{{p}} \rangle}{N}\right] \left[ \lim_{N\to \infty} \frac{\langle  W_{\ell_1} \rangle}{N}\right].$$
%The result now follows from Theorem \ref{t:plaquette} and induction. 
%
%
%By a simple translation of the loop we can assume $x_2=-1$ and hence that the matrices (of the configuration $Q$) along  the edges of the line $Y=x_2+1$ as well as the vertical lines are identity. Instead of translating the loop, we could also achieve this by a different choice of the gauge (see Section \ref{gf}). 
\end{proof}

\section{Some exact computations using asymptotic freeness.}
\label{s:comp}
In this section using the tools from free probability theory to compute certain loop statistics which seem to not easily follow from \eqref{rec10}. The proof technique will allow us to prove Proposition \ref{p:lbcounterex} in any dimension. However for readability's sake we first treat the planar case. 

We consider $\F_2,$ the free group on two generators. Let the generators be $a,b.$  Notice  the canonical correspondence between $\F_2$ and a set of loops in $\Z^2,$ formed by wrapping around  two adjacent plaquettes $a,b$ incident on the edge $e_1,$ various times (see Figure \ref{loop2}).

Consider loops of the form $\ell_{k}=aba^{-1}b^{-1} \underset{k \text{ times}}{\ldots} aba^{-1}b^{-1}$. Recall the coefficients $a_{j}(\cdot)$ from \eqref{e:series}. The following proposition computes some of them, for the loops $\ell_{k}$. 

\begin{ppn}
\label{l1} 
For any $k$ and the loop $\ell_{k}$ described above,
\begin{itemize}
\item for any $j<k$ or $j>4k,$ $a_j(\ell_k)=0;$
\item $a_{4k}=-C_{2k-1}$ where $C_k$ is the $k^{th}$ Catalan number.
\end{itemize}
\end{ppn}

Note that, as argued in Section \ref{s:area}, ${\rm area}(\ell_k)=0$. Thus Proposition \ref{p:lbcounterex} for the planar case follows directly from the above results together with the following bound on the rate of growth of the coefficients given in Lemma 10.1 of \cite{Cha15}. 

\begin{figure}[h]
\centering
\includegraphics[width=.15\textwidth]{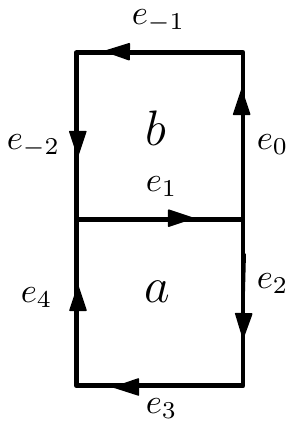}
\caption{{We call the oriented plaquette $e_1e_2e_3e_4$ as $a$ and the plaquette} $e_1e_0e_{-1}e_{-2}$ as $b.$} 
%Counter-example to the general area law lower bound is obtained by looking at powers of $aba^{-1}b^{-1}$. \red{finalize the location}}
\label{loop2}
\end{figure}
\begin{lem}
\label{growthrate100}\cite[Lemma 10.1]{Cha15} 
There exists a constant $C>0$ such that for any $j\ge 0,$  $$a_j(\ell_k) \le C^{5j}(4C)^{4k}.$$
\end{lem}

Some simple variants of the loop $\ell_k$ can be used to provide other examples of loops with nonzero area. See Remark \ref{generalarea}.

Let $\cW \in \F_2$ be a set of words formed by finite prefixes of the following ``infinite word", $$aba^{-1}b^{-1}aba^{-1}b^{-1}\ldots$$ and its automorphic images. The first part of the Proposition \ref{l1} is a consequence of the following combinatorial lemma. 

\begin{lem}
\label{min}
For any word $w\in \cW$ of length $n,$ any $\pi\in NC(n)$ with the property that all the parts of $\pi$ are singletons or pairs $(a,a^{-1}), (a^{-1},a), (b,b^{-1})(b^{-1},b)$, must have at least $n/4$ singleton parts. 
\end{lem}

\begin{proof}The proof follows by induction. The base case is easy to check. Let $n=4k+i$ where $i=0,1,2,3,$ and $$w=aba^{-1}b^{-1}aba^{-1}b^{-1}\ldots$$ have length $n.$ Consider any $\pi$ as in the statement of the lemma. Either the first $a$ is matched or it is a singleton. In the latter case the total number of singletons is one more than that coming from $\pi$ restricted  the last $n-1$ elements and hence induction can be applied. If the first $a$ is matched to some $a^{-1}$ at the $t^{th}$ position then if $t<n$ then $\pi$ be decomposed into $\pi_1$ and $\pi_2$ being restrictions on $[1,t]$ and $[t+1,n]$ respectively and again induction finishes the proof. The last remaining case is when $t=n$. Note that for that to occur $i$ must be $3.$ Thus we can just consider the restriction of $\pi$ on $[2,n-1]$. Now by hypothesis the number of singletons in this set must be at least $(n-2)/4.$ However since $i=3$ and this count is an integer , it must be at least $(n+1)/4,$ and hence we are done.
\end{proof}

We also need the following lemma which is a corollary of Lemma \ref{poly100}.

\begin{lem}
\label{refine}
Under the hypothesis of part (iii) of Lemma \ref{poly100}, for any $\pi\in NC(n)$ there exists a non crossing refinement $\pi'$ of $\pi$ with parts as in Lemma \ref{poly100} (iii), where the lowest power of $\beta$ in the  cumulant $k_{\pi'}$ (as a polynomial of $\beta$) is not larger than the corresponding term in $k_{\pi}$.   
\end{lem}

\begin{proof} 
For any $\pi\in NC(n)$ with parts $P_1,P_2\ldots$ apply Lemma \ref{poly100} (iii) for each of factors of $k_\pi(\cdot)$, namely $k_{|P_1|}(\cdot), k_{|P_2|}(\cdot), \ldots $. Since for each of the factors the lowest power of $\beta$ does not increase by the previous lemma, we are done.   
\end{proof}

We are now ready to finish the proof of Proposition \ref{l1}.

\begin{proof}[Proof of Proposition \ref{l1}]  Consider the loop $\ell_k$ and the corresponding word 
$\gamma=aba^{-1}b^{-1} \underset{k \text{ times}}{\ldots} aba^{-1}b^{-1}$. According to our notation we have $w(\ell_k,\beta)=\phi(\gamma)$  (see \eqref{e:series} and Proposition \ref{con1}). Also, by \eqref{inver1} 
\begin{equation}\label{sumcumulant}
\phi(\gamma)=\sum_{\pi \in NC(4k)}k_\pi(\gamma),
\end{equation}
where every part of $\pi$ solely contains elements from $\{a,a^{-1}\}$ or from $\{b,b^{-1}\},$ (the mixed cumulants vanish since $a,b$ are asymptotically free by Proposition \ref{con1}). Now refining each such part using Corollary \ref{refine}, we get a $\pi'$  where each part is a singleton or a pair $(a,a^{-1}), (b,b^{-1}),$   where the lowest power of $\beta$ in the  cumulant $k_{\pi'}$ (as a polynomial of $\beta$) does not increase. However for any such $\pi',$ by Lemma \ref{min}, there exists at least $k$ singletons. Hence $k_{\pi'}=\beta^{k}P(\beta)$ for some  polynomial $P(\beta).$ This completes the proof that $a_j(\ell_k)=0$ for $j<k$. That $a_j(\ell_k)=0$ for $j>4k$ is an easy consequence of Lemma \ref{poly100} (ii), and \eqref{sumcumulant}.

Note that any $\pi$ in \eqref{sumcumulant} can be decomposed into  $\pi_1$ and $\pi_2$ (the restrictions to the symbols $\{a,a^{-1}\}$ and $\{b,b^{-1}\}$ respectively). Thus both $\pi_1, \pi_2 \in NC(2k).$
Moreover $\pi_1$ and $\pi_2$ jointly satisfy the non-crossing property if and only if   $\pi_2 \preceq K(\pi_1),$ where $K(\pi_1)$ is the Kreweras complement\footnote{ It is a useful notion in free probability theory denoting the maximal connected components in the complement of the parts of $\pi_1.$ For formal definitions see Speicher \cite{speichernote}.} of $\pi_1$.
Thus using Lemma \ref{poly100} (ii), the coefficient of $\beta^{4k}$ in $\phi(\gamma)$ is 
\begin{align*}
\sum_{\pi_1 \in NC(2k)}\mu(\mathbf{0},\pi_1)\sum_{\pi_2 \in NC(2k), \,\pi_2 \preceq K(\pi_1)}\mu(\mathbf{0},\pi_2)
&=\sum_{\pi_1 \in NC(2k)}\mu(\mathbf{0},\pi_1)\mathbf{1}(\mathbf{0},K(\pi_1))\\
&=\mu(\mathbf{0}_{2k},\mathbf{1}_{2k})=-C_{2k-1},
\end{align*}
where $C_{2k-1}$ is the $(2k-1)$-th Catalan number. 
The first equality follows from \eqref{convolution}. The second follows from the well known fact that $K(\pi_1)=0$ is equivalent to $\pi_1=\mathbf{1}.$

\end{proof}

\section{Non-crossing partitions and string trajectories in any dimension}
\label{s:hd}
In this section we analyze the same loop as in the previous section for higher dimensional lattice gauges. 
The basic approach is to present a correspondence between string trajectories that appear in \cite{Cha15} and a set of non-crossing pair partitions. The inspiration comes from the proof of Lemma \ref{p:lbcounterex} in the planar case where the statement was reduced to Lemma \ref{min} (using Free Probability theory). However the aforesaid reduction continues to hold in any dimension. This completes the proof of Proposition \ref{p:lbcounterex} in any dimension.

This approach could be of general interest for understanding other loop statistics beyond the planar case. 
For any loop $\ell,$ consider a trajectory $\chi(\ell)$ starting from $\ell$ and ending at $\emptyset$, (the null string). Recall that a trajectory is a sequence of loop sequences where consecutive sequences are obtained by splitting or deformations. For more details see \cite[Corollary 3.5]{Cha15}. We shall work with fixed representations of $\ell$ and all loop sequences in $\chi(\ell)$, i.e., the loops will start and end at some fixed points and order of loops in a sequence will be fixed (one way to do this is to take the minimal representation considered by Chatterjee \cite{Cha15}, but this is unimportant for our purposes).
Let $t$ be the length of the trajectory (thought of as time). We will construct  $f_t(\ell)=e_1,e_2,\ldots,e_j,$  a sequence of edges which will not be reduced and will have many backtracks, corresponding to $\chi(\ell)$ with the following properties:\\
$\bullet$ $f_t(\ell)$ will be a closed walk i.e. the starting point of every edge $e_{i+1}$ is the ending point of $e_i;$  moreover the first point of  $e_1$ and the last point of $e_j$ would be the same as that of $\ell.$\\
$\bullet$ One can find an embedding of $\ell$ inside $f_{t}(\ell),$ i.e. there exists $1\le i_1 <i_2<\ldots < i_k \le j$  such that $\ell=e_{i_1}e_{i_2}\ldots e_{i_k}.$\\
$\bullet$ Backtrack erasure of $f_t(\ell)$ gives $\emptyset.$\\
$\bullet$ Edges of  $f_{t}(\ell)$ will be paired into partners (based on the backtrack erasure). The partner relation will also be inductively defined.

This is done by induction on the length of the trajectory. For any $\ell,$ equal to the null loop, up to backtrack erasures,  we define  $f_0(\ell)=\ell$. The partner relation is defined by paring each edge with its reversal in such a way that sequential backtrack erasure along partners (in some order) reduces $\ell$ to the null loop. Let us consider the first step in the trajectory.  We consider the following cases:

\begin{itemize}
\item For $\ell=aeb$ and positive deformation at the edge $e$ with the plaquette $p=ced.$ Then $f_{t}(\ell)=f_{t-1}(aedceb).$

\item For $\ell=aeb$ and negative deformation at the edge $e$ with the plaquette $p=ced.$ Then let $\ell_1=ac^{-1}d^{-1}b.$ Hence by induction $$f_{t-1}(\ell_1)=\gamma_1a\gamma_2c^{-1}\gamma_3d^{-1}\gamma_4b\gamma_5$$ where $\gamma_i$ are walks (recall that $\ell_1$ has an embedding in $f_{t-1}(\ell_1)$ by induction). We define $$f_{t}(\ell)=\gamma_1a ee^{-1}\gamma_2c^{-1}\gamma_3d^{-1}\gamma_4b\gamma_5.$$
Note that both the properties stated above continue to hold true by construction and induction.
We also keep track of the backtrack erasures. We  say that $\{e,e^{-1}\}$ are partners, having defined the partner relationship in $f_{t-1}(\ell_1)$ by induction. 
\item For $\ell=aebec$ and positive splitting at $e,$ let $\ell_1=aec$ and $\ell_2=eb$.  Now $\chi(\ell)$ by definition, naturally decomposes into $\chi(\ell_1)$ and $\chi(\ell_2)$ of smaller lengths $t_1$ and $t_2$ respectively. Thus by induction,  the above property of embedding of $\ell_1$ and $\ell_2$ in $f_{t_1}(\ell_1),$ and $f_{t_2}(\ell_2)$ respectively, holds. Let $f_{t_1}(\ell_1)= \gamma_1a\gamma_2e\gamma_3c\gamma_4.$ Then define $$f_{t}(\ell)=\gamma_1a f_{t_2}(\ell_2)\gamma_2e\gamma_3c\gamma_4.$$
The partners of $f_{t}(\ell)$ are the union of partners of $f_{t_1}(\ell_1)$ and $f_{t_2}(\ell_2).$
\item Negative splitting for $\ell=aebe^{-1}c,$ at $e,$ creates $\ell_1=ac$ and $\ell_2=b.$ Again like before let $f_{t_1}(\ell_1)= \gamma_1a\gamma_2c\gamma_3.$
Define, $$f_{t}(\ell)=\gamma_1ae f_{t_2}(b)e^{-1}\gamma_2c\gamma_3.$$
The partners of $f_{t}(\ell)$ are the union of partners of $f_{t_1}(\ell_1)$ and $f_{t_2}(\ell_2)$ and $\{e,e^{-1}\}.$
\end{itemize} 

It is easy to see that all the properties, $f_{t}(\ell)$ was promised to satisfy, continue to hold by induction. 

\noindent
\textbf{Observation}: The partner relationship induces a non crossing pair partition of $\{1,2,\ldots,{\rm{len}}(f_t(\ell))\}$ where ${\rm{len}}(f_t(\ell))$ denotes the length of the sequence $f_t(\ell).$
Since by construction an embedding of $\ell,$ sits inside $f_t(\ell),,$ we colour the corresponding $\rm{len}(\ell)$ points in $\{1,2,\ldots,{\rm{len}}(f_t(\ell))\},$ blue and colour everything else red.

We now restrict our attention to the blue points. Note that the partner relationship induces a non-crossing partition on only the blue points as well where all blue points with red partners form singleton parts by themselves.  
To illustrate we consider the following example: (recall the plaquettes $a$ and $b$ from Figure \ref{loop2}).

\noindent
\textbf{Example:}
Let $\ell=e_1e_2e_3e_4.$ 
We consider the following trajectory:\\
$\bullet$ Positive deformation by $b$ at $e_1.$\\
$\bullet$ Negative deformation by $b$ at $e_{-1}$\\
$\bullet$ Negative deformation by $a$ at $e_{1}$

$$f_{3}(\ell)= \textcolor{red}{e_1}\textcolor{red}{(e^{-1}_1e_{-2}^{-1}e^{-1}_{-1}e^{-1}_{0})}\textcolor{red}{e_0e_{-1}e_{-2}}\textcolor{blue}{e_1}\textcolor{red}{(e_1^{-1}e_4^{-1}e_3^{-1}e_2^{-1})}\textcolor{blue}{e_2e_3e_4}.$$

\begin{figure}[h]
\centering
\includegraphics[width=1\textwidth]{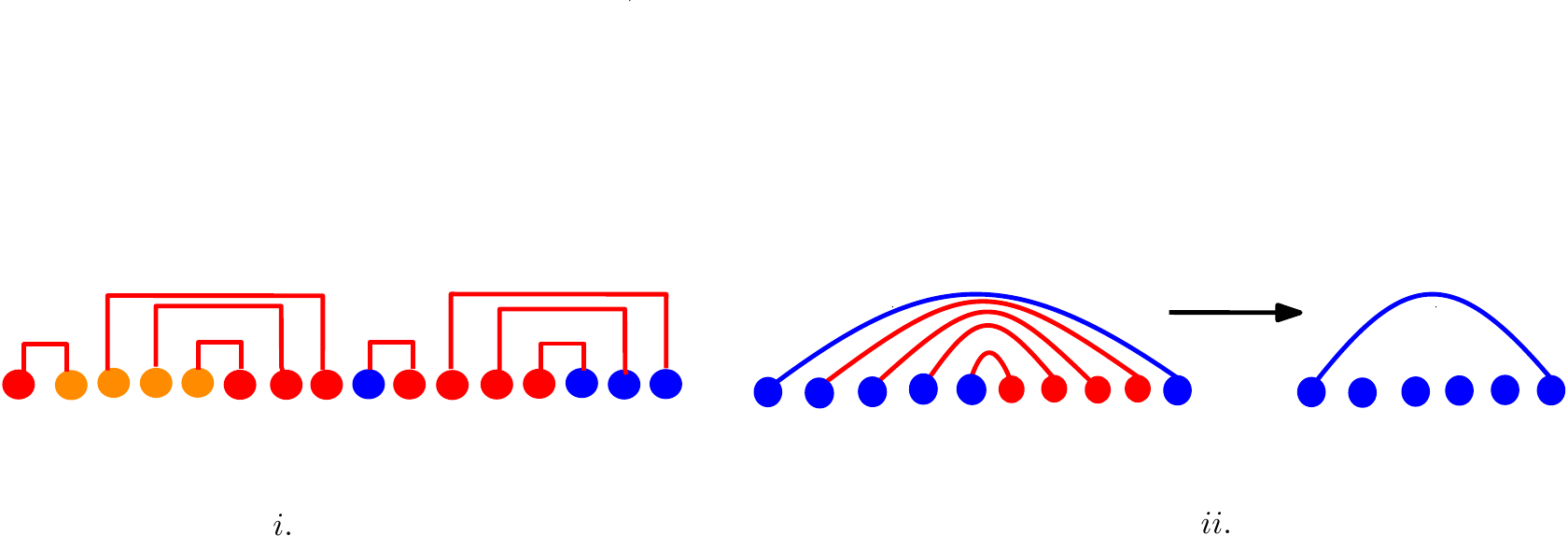}
\caption{$i.$ Shows the non-crossing partition induced by the partner relationship for the loop $f_3(\ell)$ in the example discussed above.  The blue vertices denote the embedding of $\ell$ and the edges denote the partner relationship. $ii.$ A general example of restriction of  a partner relationship to the  blue vertices.}
\label{noncross}
\end{figure}

Using the above, we now finish off the section with the following lemma, which shows that the same loop considered in Section \ref{s:comp} satisfies the conclusion of Proposition \ref{p:lbcounterex} in any dimension.
\begin{lem}\label{count1}
Consider the loop $\ell=(aba^{-1}b^{-1})(aba^{-1}b^{-1})\ldots \text{repeated}~n\text{ times}$ in $\Z^d$ where $a$ and $b$ are two adjacent plaquettes sharing a horizontal edge $e_1$. This is same as the loop $\ell_n$ from Proposition \ref{l1} but now considered in a general dimension. Then $$\lim_{N\to \infty}\frac{\langle \Tr(W_{\ell})\rangle}{N}\le (C\beta) ^{n}$$  for $\beta>0$ and for some universal constant $C$ depending only on the ambient dimension $d$.
\end{lem}
\begin{proof}
Let $a=e_1e_2e_3e_4$ and $b=e_1e_0e_{-1}e_{-2}$ as in Figure \ref{loop2}. For any trajectory $\chi(\ell),$ consider the non-crossing partition as in Figure \ref{noncross}, moreover restricted to the blue points corresponding to the occurrences of edges $e_3,e_3^{-1},e_{-1},e_{-1}^{-1}$.  Note that there are $4n$ such points.

By Lemma \ref{min}, the number of singleton parts, among this is at least $n$. This implies that there were at least $n$ red points added along $\chi(\ell)$ which were paired to the blue points. Notice that a new red point can only be added during a deformation step. Also note that any plaquette can contribute at most one edge among $e_3,e_3^{-1},e_0,e_0^{-1}$, it follows that each such red point must have been added during distinct deformation steps. Thus at least $n$ deformations were made. Observe that as shown in Lemma 10.1 of \cite{Cha15}, Lemma \ref{growthrate100} continues to hold in any dimension (the constant $K$ changes depending on the ambient dimension) and hence we are done as in the planar case.  
\end{proof}

\begin{remark}\label{generalarea} Note that the above method of proof is robust and  can be used to show that for any other loop $\gamma$ which differs from $\ell$ by at most (say) $n/10$ edges, any vanishing trajectory from $\gamma$ must need at least $n-n/5$ many deformations. 
\end{remark}

\section{Concluding Remarks and Open questions}
\label{s:oq}
In this paper we have evaluated explicitly certain Wilson loop expectations in the strongly coupled planar $SO(N)$ Lattice gauge theory following the fundamental work of Chatterjee \cite{Cha15}, using some combinatorial analysis and connections with free probability. A natural further direction is to investigate if Chatterjee's recursion can be used to conclude about loop statistics in higher dimensions. For example, whether there are random surface  (randomness induced by the appropriate Gibbs measure) analogues of decorated trees  encountered in Section \ref{pla}, representing trajectories in higher dimension, remains a very interesting question. 

We saw that the area law lower bound, in the planar setting, was true for simple loops (see Corollary \ref{c:arealb}) but not quite true in general (see Proposition \ref{p:lbcounterex}) in any dimension. The latter was proved by constructing a loop  with bounded area but which required a large number of deformations to be reduced to the null loop. 
A natural question (pointed out by Chatterjee) therefore  is whether a version of Corollary \ref{c:arealb} is true in general for any $\ell$ if one replaces the $\rm{area}(\ell)$ term in the exponent in the lower bound  by the minimum number of deformations needed to reduce $\ell$ to the null loop.

Through \cite[Corollary 3.4]{Cha15} understanding of the plaquette statistics can be used to understand the log-partition function for the Lattice Gauge Gibbs measure. Thus Theorem \ref{pla100} implies 
$$\lim_{N\to \infty} \frac{\log Z_{\Lambda_N,N,\beta}}{N^2 |\Lambda_N|}=\frac{\beta^2}{2}$$ in the strong coupling regime, ($Z_{\Lambda_N,N,\beta}$ is the partition function for the Gibbs measure with inverse coupling constant $\beta$, matrices in $SO(N)$ and the box $\Lambda_N=[-M_N,M_N]^2$ where $M_N$ is a sequence of positive integers increasing to infinity). 
For $U(N)$ Lattice Gauge theory, using mean field approximations to eigenvalue distributions, Wadia \cite{Wadia} does a computation predicting the log partition function throughout the entire regime of $\beta$ which is supposed to exhibit a phase transition at $\beta=1,$  established rigorously by \cite{LDPunitary}.  %Another direction is to move beyond the Gauge group $SO(N)$. 

Physics literature also predicts that up to some re-parametrization the large $N$ limit of the Wilson loop expectations should be same for several different Gauge groups. Our conclusions in the planar regime agree with the predictions by Wadia \cite{Wadia} and Gross-Witten \cite{GW80} for the $U(N)$ lattice Gauge theory with the inverse coupling constant $2\beta$. As already mentioned in Remark \ref{reparam}, for $SU(N)$ lattice gauge theory, (where the trace in the Hamiltonian in \eqref{e:gibbs} is replaced by the real part of the trace) on $\Z^d$ for any $d\geq 2$, the same was established by Jafarov \cite{JJ16}, i.e., the large $N$ limit of loop expectation for a fixed loop $\ell$ at the inverse coupling constant $2\beta$ for the  $SU(N)$ lattice gauge theory is the same as the limiting loop expectation for $\ell$ at inverse coupling strength $\beta$ in the $SO(N)$ lattice gauge theory. In particular this implies our conclusions regarding loop expectations will all be valid for  $SU(N)$ lattice gauge theory too, in the 't Hooft limit after an appropriate change of variable. 

We finish by sketching a proof of Proposition \ref{p:lsd}.
Observe that under the aforementioned change of variable $\beta \mapsto 2\beta,$ this agrees with the limiting eigenvalue distribution obtained in \cite{LDPunitary}. 
% predicted by the mean-field calculation of Gross-Witten \cite{GW80} and Wadia \cite{Wadia}. We shall only provide a sketch of the proof.

\begin{proof}[Proof of Proposition \ref{p:lsd}]
As mentioned before, we rely on the method of moments. Recall that $F_{N}$ denotes the empirical spectral distribution of a plaquette variable $Q_{p}$ (for $N$ sufficiently large so that the plaquette $p$ is contained in $\Lambda_{N}$). Clearly $F_{N}$ is supported on $S^1$. Let $\langle F_{N} \rangle$ denote the expected empirical spectral measure, (i.e., for any bounded continuous $f$, $\int f {\rm d}\langle F_{N}\rangle =\langle \int f {\rm d} F_{N} \rangle$). Let $\Theta_{N}$ denote a random variable taking values in $[0,2\pi)$ such that $e^{i\Theta_{N}}$ has the same distribution as $\langle F_{N} \rangle$. For this proof, let $\E$ denote the expectation with respect to $\langle F_{N}\rangle $, and hence, $\E e^{ik\Theta_{N}}=\frac{1}{N}\langle \Tr Q_p^{k}\rangle$, for all integer $k$. Taking real parts, it follows from Theorem \ref{pla100}, that as $N\to \infty$, $\langle {\E} \cos (k\Theta_N) \rangle\to \beta$ if $k=1$ and zero for all larger integer values of $k$. By a simple calculation, one can check that this determines the limiting moment sequence $ \E \cos^k \Theta_{N}$ as follows: 
$$\lim_{N\to \infty} {\E} \cos^k (\Theta_N) = \left \{ \begin{array}{cc}
\dfrac{\binom{k}{\frac{k-1}{2}}}{2^{k-1}}\beta; &  k~ \text{is odd},\\
\dfrac{\binom{k}{\frac{k}{2}}}{2^{k}}; & k~\text{is even}.
 \end{array}
 \right.$$
As the variable $\cos \Theta_{N}$ as these variable are uniformly bounded, the moment sequence determines the limiting distribution of $\cos \Theta_{N}$. Now observe that the eigenvalues of $SO(N)$ occur in complex conjugate pairs and hence the limiting distribution of $\Theta_{N}$ is completely determined the limiting distribution of $\cos \Theta _{N}$. Lastly, by another easy calculation, it follows that  the moment sequence of $\cos \Theta$ where $\Theta$ is distributed according to $F$ matches the limiting moment sequence obtained above. This completes the proof of weak convergence of $\langle F_{N} \rangle$ to $F$. To upgrade this to convergence in probability it suffices to show that the variance (with respect to the Lattice Gauge measure of $\frac{1}{N} \Tr Q_p^{k}$ converges to 0 for each $k$. This is a straightforward consequence of of the factorization theorem (Theorem \ref{t:factor}) and we omit the details.
%Note that in this case, to conclude convergence in probability of the empirical spectral measure to $F,$ one has to show that the empirical spectral moments  $\textcolor{red}{\E} \cos (k\Theta_N)$ are concentrated near $\langle \textcolor{red}{\E} \cos (k\Theta_N) \rangle.$ However, as a straightforward consequence of the factorization theorem (Theorem \ref{t:factor}), one obtains that the variance of $\textcolor{red}{\E} \cos (k\Theta_N)$ converges to zero and hence we are done.\ref{}\textcolor{red}{say something about moving to convergence in probability}. 
\end{proof}

%\textcolor{red}{make a comment about the $\beta=0$ case. That comment has already been made before.} 

For an elaborate list of related  open questions, see \cite{Cha15}.

%\section{Disjoint loops are asymptotically free}

\section*{Acknowledgements}
We are extremely grateful to Sourav Chatterjee for many illuminating discussions and helpful comments on a preliminary version of the paper. We also thank an anonymous referee for useful comments and suggestions. R.B.'s research is partially supported by an AMS-Simons Travel Grant, S.G.'s research is supported by a Miller Research Fellowship at UC Berkeley.

\bibliography{latticegauge}
\bibliographystyle{plain} 
\end{document}